\documentclass[a4paper,11pt,reqno]{amsart}
\usepackage[centertags]{amsmath}
\usepackage{amsfonts,amssymb,amsthm} 
\usepackage[dvips]{graphicx} 
\usepackage{psfrag}
\usepackage[english]{babel}
\usepackage{newlfont}
\usepackage{mathrsfs}
\usepackage{color}
\usepackage[body={16cm,22.7cm},centering]{geometry} 
\usepackage[colorlinks,citecolor=blue,urlcolor=blue]{hyperref}
\usepackage{fancyhdr}
\newtheorem*{mtheorem}{Main Theorem}
\newtheorem*{conjecture}{Conjecture}
\newtheorem{lemma}{Lemma}
\newtheorem{remark}{Remark}
\newtheorem{proposition}{Proposition}
\newtheorem{corollary}{Corollary}

\newcommand{\e}{\varepsilon}

\newcommand{\R}{\mathbb{R}}

\newcommand{\Ha}{\mathcal{H}}
\newcommand{\Hh}{\mathscr{H}} 
\newcommand{\Os}{\mathcal{E}}
\newcommand{\Sf}{\mathbb{S}}
\newcommand{\F}{\mathcal{F}}
\newcommand{\B}{\mathcal{B}}
\renewcommand{\div}{\mathrm{div}}
\newcommand{\loc}{\mathrm{loc}}
\newcommand{\la}{\langle}
\newcommand{\ra}{\rangle}
\newcommand{\unit}{\omega}
\mathchardef\emptyset="001F

\title[Gaussian isoperimetric inequality]
{Sharp dimension free quantitative estimates\\ for the Gaussian isoperimetric inequality}
\author[M. Barchiesi, A. Brancolini, V. Julin]
{Marco Barchiesi, Alessio Brancolini, Vesa Julin}
\address[M. Barchiesi]{Universit\`{a} di Napoli ``Federico II'',
Dipartimento di Matematica e Applicazioni,
Via Cintia, Monte Sant'Angelo, I-80126 Napoli, Italy}
\email{barchies@gmail.com}
\address[A. Brancolini]{University of M\"{u}nster,
Institute for Numerical and Applied Mathematics,
Einsteinstra\ss e 62, D-48149 Germany}
\email{alessio.brancolini@uni-muenster.de}
\address[V. Julin]{University of Jyv\"{a}skyl\"{a},
Department of Mathematics and Statistics,
P.O.Box 35 (MaD) FI-40014, Finland}
\email{vesa.julin@jyu.fi}
\date{\today}

\begin{document}
\maketitle

%%%%%%%%%%%%%%%%%%%%%%%%%%%%%%%%%%%%%%%%%%%%%%%%%%%%%%%%%%%%%%%%%%%%%%%%%%%%%%%%%%%%%%%%%%%%%%
\begin{center}
\begin{minipage}{13cm}
\small{
\noindent {\bf Abstract.} 
We provide a full quantitative version of the Gaussian isoperimetric inequality:
the difference between the Gaussian perimeter of a given set and
a half-space with the same mass controls the gap between the norms of the corresponding 
barycenters. In particular, it controls the Gaussian measure of the symmetric difference 
between the set and the half-space oriented so to have the barycenter
in the same direction of the set.
Our estimate is independent of the dimension, sharp on the decay rate with respect to 
the gap and with optimal dependence on the mass. 

}
\end{minipage}
\end{center}

\bigskip

%%%%%%%%%%%%%%%%%%%%%%%%%%%%%%%%%%%%%%%%%%%%%%%%%%%%%%%%%%%%%%%%%%%%%%%%%%%%%%%%%%%%%%%%%%%%%%%%%%%%%
%\tableofcontents

\section{Introduction}

\noindent
The isoperimetric inequality in Gauss space states that among all sets with a given Gaussian measure 
the half-space has the smallest Gaussian perimeter. 
This result was first proved by Borell~\cite{Bor} and  independently by Sudakov-Tsirelson \cite{SuCi}. 
Since then  many alternative proofs have been proposed, e.g. \cite{BL, Bob, Eh1}, but the issue of completely characterizing  
the extremals was settled only more recently by Carlen-Kerce \cite{CK}, establishing that half-spaces are the unique solutions
to the Gaussian isoperimetric problem.

The natural issue of proving a quantitative version of the isoperimetric inequality 
turns out to be  a much more delicate task. An estimate in terms of  the Fraenkel asymmetry, i.e., the Gaussian measure of the 
symmetric difference between a given set and a half-space, was recently established by Cianchi-Fusco-Maggi-Pratelli \cite{CFMP}.
This result provides the sharp decay rate with respect to the Fraenkel asymmetry but with a non-explicit, 
dimensionally dependent constant. As for the analogous result in the groundbreaking paper in the
Euclidean space  \cite{FMP}, the proof is purely geometric and is 
based on a reflection argument  in order to  reduce the problem to sets which are $(n-1)$-symmetric.  This will cause 
the constant to blow up at least exponentially with respect to the dimension. 
However, the fact that in Gauss space  most geometric and functional inequalities  are  
independent of the dimension suggests that such a quantitative version of the Gaussian isoperimetric inequality should also
be dimension free. This would also be important for possible applications, see \cite{MDO, MN, MN2} and the references therein.
Indeed, after \cite{CFMP},  Mossel-Neeman \cite{MN, MN2} and  Eldan \cite{EL} have provided  quantitative estimates which are 
dimension free but have a sub-optimal decay rate with respect to the Fraenkel asymmetry. 
It is therefore a natural open problem whether a quantitative estimate holds with a sharp
decay rate and, simultaneously, without dimensional dependence.

In this paper we  answer affirmatively to this question. Our result is valid not only for the Fraenkel asymmetry
but for a stronger one  introduced in \cite{EL} which measures the difference of the barycenter of a given set from 
the barycenter of a half-space. 
Our quantitative isoperimetric inequality is completely explicit, and it also has the optimal dependence on the mass. 
The main result is given in terms of the strong asymmetry since in our opinion this is a  more natural way 
to measure the stability of the  Gaussian isoperimetric inequality.  
We will also see that the strong asymmetry appears 
naturally when one considers an asymmetry which we call  the \emph{excess} of the set. This is the Gaussian 
counterpart of the oscillation asymmetry in the Euclidean setting introduced by Fusco and the third author
in \cite{FJ} (see also \cite{BDF, BDS}).

Subsequent to \cite{FMP}, different proofs in the Euclidean case have been given in~\cite{FigMP} 
(by the optimal transport) and in \cite{AFM, CL} (using the regularity theory for minimal surfaces 
and the selection principle). 
Both of these strategies are rather flexible and have been adopted to prove many other geometric inequalities
in a sharp quantitative form. Nevertheless,  they do not 
seem  easily  implementable for our purpose. Indeed, it is not known if the Gaussian  
isoperimetric inequality itself can be retrieved from optimal transport (see \cite{Vi}). 
On the other hand, the approach via selection principle 
is by contradiction. Therefore,  if it may be adapted to the Gaussian setting, 
it cannot be used  as it is to provide explicit information about the constant in the quantitative 
isoperimetric inequality. Finally, the proof in \cite{EL} is based on stochastic calculus and provides sharp 
estimates for the Gaussian noise stability inequality. As a corollary, this gives a quantitative estimate for 
the Gaussian isoperimetric inequality which is, however, not sharp. In order to prove the sharp quantitative 
estimate we introduce a  technique which is based on a direct analysis of the first and the second variation conditions of 
solutions to a suitable minimization problem. This enables us to obtain the  sharp result with a very  
short proof. We will outline the proof  at the end of the Introduction. 

In order to describe the problem  more precisely, we introduce our setting. 
Given a Borel set $E\subset\R^n$, $\gamma(E)$ denotes its \emph{Gaussian measure}, defined as
\begin{equation*}
\gamma(E):=\frac{1}{(2\pi)^\frac{n}{2}}\int_E e^{-\frac{|x|^2}{2}}dx.
\end{equation*}
If $E$ is an open set with Lipschitz boundary, $P_\gamma(E)$ denotes its \emph{Gaussian perimeter}, defined as 
\begin{equation}\label{gauss perim}
P_\gamma(E):=\frac{1}{(2\pi)^\frac{n-1}{2}}\int_{\partial E}e^{-\frac{|x|^2}{2}}d\Ha^{n-1}(x),
\end{equation}
where $\Ha^{n-1}$ is the $(n-1)$-dimensional Hausdorff measure.
Moreover, given $\unit \in\Sf^{n-1}$ and $s\in\R$,  $H_{\unit,s}$ denotes the half-space of the form
\begin{equation*}
H_{\unit,s}:=\{x\in\R^n \text{ : } x\cdot\unit<s\}.
\end{equation*}
We define also the function $\phi:\R\rightarrow(0,1)$ as 
\begin{equation*}
\phi(s):=\frac{1}{\sqrt{2\pi}}\int_{-\infty}^s e^{-\frac{t^2}{2}}dt.
\end{equation*}
Then we  have $\gamma(H_{\unit,s})=\phi(s)$ and $P_\gamma(H_{\unit,s})=e^{-s^2/2}$.
The isoperimetric inequality states that, given an open set $E$ with Lipschitz boundary and mass
$\gamma(E)=\phi(s)$, one has
\begin{equation}\label{iso ine}
P_\gamma(E)\geq e^{-s^2/2},
\end{equation}
and the equality holds if and only if $E=H_{\unit,s}$ for some $\unit \in\Sf^{n-1}$.

A natural question is the stability of the inequality \eqref{iso ine}. Let us denote by $D(E)$
the Gaussian isoperimetric deficit (i.e., the gap between the two sides of the isoperimetric inequality),
\begin{equation*}
D(E):=P_\gamma(E)-e^{-s^2/2},
\end{equation*}
and by $\alpha(E)$ the Fraenkel (or the standard) asymmetry,
\begin{equation*}
\alpha(E):=\min_{\unit\in\Sf^{n-1}}\gamma(E\triangle H_{\unit,s}),
\end{equation*}
where $\triangle$ stands for the symmetric difference between sets. As we mentioned, it is proved in \cite{CFMP} 
that for every set $E \subset \R^n$ with $\gamma(E)= \phi(s)$ the isoperimetric deficit controls the square of 
the Fraenkel asymmetry, i.e.,
\begin{equation}\label{stato arte}
\alpha(E)^2\leq c(n,s)D(E),
\end{equation}
and  the exponent $2$ on the left-hand side is sharp.  
On the other hand, in \cite{MN2} a similar estimate  is proved (for $s=0$), with
a sub-optimal exponent  but with a constant independent of the dimension.
The following natural conjecture is stated explicitly in \cite[Conjecture 1.8]{MN2}
(see also \cite[Open problem 6.1]{MN} and the discussion in \cite{EL}).
\begin{conjecture}
Inequality \eqref{stato arte} holds for a constant $c(s)$ depending only on the mass $s$.  
\end{conjecture}

In \cite{EL} Eldan introduces  a new asymmetry which is equivalent to 
\begin{equation}\label{stronger}
\beta(E):=\min_{\unit\in\Sf^{n-1}}\big|b(E)-b(H_{\unit,s})\big|,
\end{equation}
where 
\begin{equation*}
b(E):=\int_E x \, d\gamma(x)
\end{equation*} 
is the (non-renormalized) barycenter of the set $E$,  and $s$ is chosen 
such that $\gamma(E)=\phi(s)$. We call this \emph{strong asymmetry} since it controls the standard one as 
(see Proposition \ref{standard vs strong})
\begin{equation} \label{from prop 4}
\beta(E)\geq  \frac{e^{\frac{s^2}{2}}}{4} \,\alpha(E)^2.
\end{equation}
In \cite[Corollary 5]{EL} it is proved that
\begin{equation}\label{EL log}
\beta(E) \big| \log \beta(E)  \big|^{-1} \leq c(s)D(E)
\end{equation}
for an inexplicit constant $c(s)$ depending only on $s$. Together with \eqref{from prop 4}, this proves the
conjecture  up to a logarithmic factor. Estimate \eqref{EL log} is derived by the so-called robustness estimate for 
the Gaussian noise stability, where the presence of the logarithmic term cannot be avoided
(see \cite[Theorem 2 and discussion in Section 1.1]{EL}). 

\vspace{4pt}
In this paper we fully prove the conjecture. In fact, we prove an even  stronger result, since we provide the optimal 
quantitative estimate in terms of the strong asymmetry. Our main result reads as follows.
\begin{mtheorem}
There exists an absolute constant $c$ such that for every $s \in \R$ and for every set $E\subset\R^n$  with $\gamma(E)=\phi(s)$ 
the following estimate holds:
\begin{equation}\label{largo gente}
\beta(E)\leq c\,(1+s^2) D(E).
\end{equation}
\end{mtheorem}

\vspace{4pt}
In Remark \ref{optimal_mass} we show that the dependence on the mass is optimal.
This can be seen by comparing a one-dimensional interval $(-\infty,s)$ with  a union of two intervals 
$(-\infty, -a) \cup (a , \infty)$ with the same Gaussian length. 
Concerning the numerical value of the constant $c$, we show that we may consider
\begin{equation*}
c=12\sqrt{2\pi},
\end{equation*}
 which  is not optimal. From \eqref{from prop 4} and \eqref{largo gente} we immediately
conclude that for every set  $E\subset\R^n$  with $\gamma(E)=\phi(s)$ the following improvement of \eqref{stato arte} holds:
\begin{equation*}
\alpha(E)^2 \leq 4c\,(1+s^2)e^{-\frac{s^2}{2}}\,   D(E).
\end{equation*}
Finally, since the decay rate with respect to the Fraenkel asymmetry in \eqref{stato arte} is sharp this implies
that also the linear dependence on $\beta(E)$ in \eqref{largo gente} is sharp.

We may state the result of the Main Theorem in a more geometrical way. 
Define for a given (sufficiently regular) set $E$ its \emph{excess} as
\begin{equation*}
\Os(E):=\min_{\unit \in\Sf^{n-1}}\left\{
\frac{1}{(2\pi)^{\frac{n-1}{2}}}\int_{\partial E}|\nu^E- \unit|^2 e^{-\frac{|x|^2}{2}} \, d\Ha^{n-1}(x)\right\},
\end{equation*}
where $\nu^E$ is the exterior normal of $E$.  In Corollary \ref{osci2} at the end of Section \ref{quanti est} 
we show that for every set $E$ it holds
\begin{equation*}
\Os(E) =2D(E)+2\sqrt{2\pi}\beta(E).
\end{equation*}
Therefore by the Main Theorem we conclude that the deficit controls also the excess of the set.
Roughly speaking, this means  that the closer the perimeter of $E$ is to the perimeter of half-space, 
the flatter its boundary has to be. This is the Gaussian counterpart of the result in \cite{FJ} for the Euclidean case,
and it highlights the importance of the strong asymmetry.

As we already mentioned, the proof of the  Main Theorem is based on a direct variational method. 
The idea is to write the inequality \eqref{largo gente} as a  minimization problem 
\begin{equation*}
\min \big\{ P_\gamma(E)+  \e |b(E)|^2  \, : \,\, \gamma(E) = \phi(s) \}
\end{equation*}
and deduce directly from  the first  and the second variation conditions that when $\e>0$ is small enough the 
only solutions are half-spaces. It is not difficult to see that this  is equivalent to the statement of the Main Theorem.  
In Section \ref{sect min} we study the regularity of the solutions to the above problem, derive the Euler equation 
(i.e., the first variation is zero) and  the second variation condition. 
In Section \ref{quanti est} we give the proof of the Main Theorem. 
The key point of the proof is a careful choice of test functions in the second variation condition,
which permits to conclude directly that when $\e$ is sufficiently small every minimizer is a union of parallel stripes.
Since this is true in every dimension and the choice of $\e$ does not depend on $n$, this argument reduces the problem 
to the one-dimensional case. We give a more detailed overview of the proof in Section \ref{idea}. 
Finally, we would like to mention recent works \cite{MR, Ro} where the authors  use the  second variation condition 
to study isoperimetric inequalities in Gauss space.

%%%%%%%%%%%%%%%%%%%%%%%%%%%%%%%%%%%%%%%%%%%%%%%%%%%%%%%%%%%%%%%%%%%%%%%%%%%%%%%%%%%%%%%%%%%%%%%%%%%%%%%%%%%%%%%%%%%%%%%%
\section{Notation and preliminaries}\label{notation}

\noindent
In this section we briefly introduce our basic notation and recall some elementary results from geometric measure theory. 
For an introduction to the theory of sets of finite perimeter we refer to \cite{AFP} and \cite{Ma}.  

We denote by $\{e^{(1)},\ldots,e^{(n)}\}$  the canonical base of $\R^n$. For  generic point $x \in \R^n$ we denote its  
$j$-component by $x_j := \la x, e^{(j)} \ra$ and  use the notation $x =(x',x_n)$ when we want to specify the 
last component. Throughout the paper  $B_R(x)$ denotes the open ball centered at $x$   with  radius $R$.
When the ball  is centered at the origin we simply  write $B_R$.
The family of the Borel sets in $\R^n$ is denoted by $\B$ . We denote the $(n-1)$-dimensional
Hausdorff measure with Gaussian weight by $\Ha^{n-1}_\gamma$, i.e.,  for every set $A \in \B$ we define
\begin{equation*}
\Ha^{n-1}_\gamma(A):=\frac{1}{(2\pi)^\frac{n-1}{2}}\int_{A}e^{-\frac{|x|^2}{2}}d\Ha^{n-1}(x).
\end{equation*}

A set $E \in \B$ has \emph{locally finite perimeter} if $\chi_E\in BV_{\loc}(\R^n)$, i.e.,
for every ball $B_R \subset \R^n$ it holds
\[
\sup \Big\{ \int_E \div \varphi\, dx \,\, : \,\, \varphi \in C_0^\infty(B_R; \R^n), \, \sup|\varphi|\leq 1 \Big\} < \infty.
\]
If $E$ is a set of  locally finite perimeter, we define the \emph{reduced boundary}
$\partial^*E$ of $E$ as the set of all points $x\in\R^n$ such that
\begin{equation*}
\nu^E(x):= -\lim_{r\rightarrow0^+}\frac{D\chi_E(B_r(x))}{|D\chi_E|(B_r(x))} 
\quad\text{ exists and belongs to } \Sf^{n-1}.
\end{equation*}
The reduced boundary $\partial^*E$ is a subset of the topological boundary $\partial E$ and coincides, up to a 
$\Ha^{n-1}$-negligible set, with the support of $D\chi_E$. When $E$ is an open set with Lipschitz boundary then 
$\Ha^{n-1}(\partial E \Delta \partial^*E) = 0$ \cite[Example 12.6]{Ma}.
We shall refer to the vector $\nu^E(x)$ as the \emph{(generalized) exterior normal}
at $x \in \partial^* E$. For more information we refer to \cite[Definition 3.54]{AFP}.
When no confusion arises we shall simply write $\nu$ and use the notation 
$\nu_j =  \la \nu, e^{(j)} \ra$.  
If $E$ has locally finite perimeter, then its perimeter in  $A \in \B$ is
\[
P(E; A) := \Ha^{n-1}(\partial^* E \cap A).
\]
Moreover,  by the divergence theorem we have
\[
\int_E \div X\, dx = \int_{\partial^* E} \la X, \nu^E \ra\, d\Ha^{n-1}(x)
\] 
for every Lipschitz continuous vector field $X: \R^n \to \R^n$ with compact support. 

In \eqref{iso ine} the Gaussian isoperimetric problem was stated for sets with Lipschitz boundary, but this can be extended
to  more general and more natural class of sets.  Indeed,  if $E \in \B$  is a set of locally finite perimeter with 
$\Ha^{n-1}_\gamma(\partial^* E)< \infty$, then it has \emph{finite Gaussian perimeter} and we denote its Gaussian perimeter by
\[
P_\gamma(E) := \Ha^{n-1}_\gamma(\partial^* E).
\]  
Otherwise we set $P_\gamma(E) := \infty$. It follows from  the divergence theorem that
\begin{equation}\label{g perimeter}
P_\gamma(E)  = \sqrt{2 \pi}\sup \Big\{ \int_E (\div \varphi - \la \varphi , x \ra)\, d\gamma(x) \,\, : \,\, \varphi \in C_0^\infty(B_R; \R^n), \, \sup|\varphi|\leq 1 \Big\}
\end{equation}
for every  $E \in \B$. If not otherwise specified, throughout we assume that every set has finite Gaussian perimeter.
The above notion of Gaussian perimeter provides an extension of~\eqref{gauss perim}
because, if $E$ is an open set with Lipschitz boundary, then $\partial E$ and $\partial^*E$ coincide up to 
a $\Ha^{n-1}$-negligible set.

We recall  some notation for calculus on smooth hypersurfaces  (see \cite[Section 11.3]{Ma}).
Let us fix a set $E \subset \R^n$ and assume that there is an open set $U \subset \R^n$ such that 
$M= \partial E \cap U$ is a $C^\infty$ hypersurface.
Assume that we have a vector field $X \in C^\infty(M;\R^n)$. Since the manifold $M$ is smooth we may 
extend $X$ to  $U$ so that  $X \in C^\infty(U;\R^n)$. We define the tangential differential of $X$ on $M$ by
\[
D_\tau X(x) := DX(x) - (DX(x) \nu^E(x)) \otimes \nu^E(x) \qquad  x \in M,
\] 
where $\otimes$  denotes the tensor product.  It is clear that $D_\tau X$ depends only on the values of 
$X$ at $M$, not on the chosen extension. The tangential divergence of $X$ on $M$ is defined by
\[
\div_\tau X := \text{Trace}(D_\tau X) =  \div X - \la DX \nu^E, \nu^E \ra.
\]
Similarly, given  a function $u \in C^\infty(M)$  we extend it to $U$ and define its tangential gradient by
\[
D_\tau u := D u  -  \la Du, \nu^E \ra\, \nu^E.
\]
We define the tangential derivative of $u$ in direction $e^{(i)}$ as
\[
\delta_i u := \la D_\tau u,  e^{(i)} \ra =  \partial_{x_i} u - \la \nabla u, \nu \ra \nu_i.
\]
The tangential Laplacian of $u$ on $M$ is
\[
\Delta_\tau u := \div_\tau (D_\tau u) = \sum_{i=1}^n \delta_i ( \delta_i u).
\]
Since $M$ is smooth, the  exterior normal is a smooth vector field $\nu^E  \in C^\infty(M;\R^n)$. 
Then the sum $\Hh(x)$ of the principal curvatures at $x \in M$ is given by
\[
\Hh(x)= \div_\tau (\nu^E(x)). 
\]
We denote by $|B_E|^2$  the sum of the squares of the principal curvatures, which can be written as
\[
|B_E|^2 = \text{Trace}(D_\tau \nu^E D_\tau \nu^E  ) = \sum_{i,j=1}^n (\delta_i \nu_j)^2.
\]
Note that  $D_\tau \nu^E$ is symmetric, i.e., $\delta_i \nu_j= \delta_j \nu_i$ (see \cite[formula (10.11)]{Giusti}).  
Finally, the Gauss-Green theorem, or the divergence theorem, on hypersurfaces  states that for every 
$X \in C_0^\infty(M; \R^n)$ it holds
\[
\int_M \div_\tau X\, d\Ha^{n-1}(x) = \int_M \Hh \la X, \nu^E \ra \, d\Ha^{n-1}(x).
\]

%%%%%%%%%%%%%%%%%%%%%%%%%%%%%%%%%%%%%%%%%%%%%%%%%%%%%%%%%%%%%%%%%%%%%%%%%%%%%%%%%%%%%%%%%%%%%%%%%%%%%%%%%%%%%%%%%%%%%%%%%%%%
\section{Overview of the proof}\label{idea}

\noindent
As we wrote in the Introduction, we will derive our main estimate \eqref{largo gente} by a suitable minimization problem.
To this aim, given $\e>0$ and $s\leq0$, we consider the functional
\begin{equation*}
\mathcal{F} (E) =  P_\gamma(E)+\e\sqrt{\pi/2}\,|b(E)|^2, \qquad \gamma(E)= \phi(s).
\end{equation*}
In fact, in the proof we replace the volume constraint by a volume penalization, but this is of little importance.  
For simplicity, we will indicate by $b_s$ the norm of the barycenter $b(H_{\unit,s})$, since it does not depend on $\unit$. 
We have $b(H_{\unit,s})=-b_s \unit$ and $b_s=e^{-\frac{s^2}{2}}/\sqrt{2\pi}$.
It is important to observe that the half-spaces maximize the norm of the barycenter, 
\begin{equation}\label{b_s}
b_s\geq|b(E)|
\end{equation}
for every set $E$ such that $\gamma(E)=\phi(s)$. 
Indeed, if $b(E)\neq0$, by taking $\unit=-b(E)/|b(E)|$, we have
\begin{equation*}\begin{split}
|b(E)|-b_s&=\langle b(E)+b_s \unit,-\unit\rangle
=-\int_E\langle x,\unit\rangle d\gamma(x)+\int_{H_{\unit,s}}\langle x,\unit\rangle d\gamma(x)\\
&=\int_{E\setminus H_{\unit,s}}(\langle x,-\unit\rangle+s) d\gamma(x)+\int_{H_{\unit,s}\setminus E}(\langle x,\unit\rangle-s) d\gamma(x)
\leq 0,
\end{split}\end{equation*}
because the integrands in the last term are both negative. This enlightens the fact that in 
minimizing $\F$ the two terms $P_\gamma(E)$ and $|b(E)|$ are in competition.  
Minimizing $P_\gamma(E)$ means to push the set $E$ at infinity in one direction, so that it becomes 
closer to a half-space. On the other hand, minimizing $|b(E)|$ means to balance the mass of $E$ with respect to the origin. 
We will see, and this is the \emph{main point of our analysis}, that for $\e$ small enough the perimeter term overcomes
the barycenter, and the only  minimizers of $\F$ are the half-spaces $H_{\unit,s}$.

We have observed that  the half-spaces maximize the norm of the barycenter.
When $b(E)\neq 0$, the minimum in \eqref{stronger} is attained by $\unit=-b(E)/|b(E)|$ and with this choice of $\unit$ we have
\begin{equation*}
\beta(E)=|b(E)+b_s\unit|=\left|\left(-|b(E)|+b_s\right)\unit\right|
=b_s-|b(E)|.
\end{equation*}
Therefore the strong asymmetry is nothing else than the gap between the maximum $b_s$ and the norm of $b(E)$.
If we show that for some $\e$ and $\Lambda$ (only depending on $s$) the only minimizers of the functional $\F$ are 
the half-spaces~$H_{\unit,s}$, $\unit\in\Sf^{n-1}$, we are done, since this implies that for every set $E \subset \R^n$  
with $\gamma(E)= \phi(s)$ it holds
\begin{equation}\label{here we are}\begin{split}
D(E)&\geq \e\sqrt{\frac{\pi}{2}}\,\bigl(b_s^2-|b(E)|^2\bigr)=\e\sqrt{\frac{\pi}{2}}\,\bigl(b_s+|b(E)|\bigr)\beta(E)\\
&\geq\frac{\e}{2}\,e^{-\frac{s^2}{2}}\beta(E).
\end{split}\end{equation}

\vspace{6pt}
Since the proof  involves many technicalities, we will carry out a sketch of the argument in order to enlighten the core ideas. 
The proof is  divided in two parts. First we prove  standard results concerning the minimizers of $\F$, such as the existence 
and the regularity of minimizers and derive the Euler equation and the second variation condition.  The existence of a minimizer  
follows directly from a compactness argument using  the lower semicontinuity of the Gaussian perimeter.
The regularity is a consequence of the regularity theory for almost minimizers of the perimeter.

The derivation of the Euler equation is  standard  but we  prefer to  sketch  the argument here. 
Let $E$ be a minimizer of $\F$ and assume that its boundary is a smooth hypersurface. 
Given a  function $\varphi \in C^\infty(\partial  E)$ with zero average, 
$\int_{\partial E}\varphi \, d\Ha^{n-1}_\gamma(x)= 0$, we choose a specific vector field  $X: \R^n \to \R^n$  
such that  $X:= \varphi \nu^E$ on $\partial E$.
Let $\Phi : \R^n \times(-\delta,\delta) \to \R^n$ be the flow associated with $X$, i.e.,   
\begin{equation*}
\frac{\partial}{\partial t} \Phi(x,t) = X( \Phi(x,t) ), \qquad \Phi(x,0) = x.
\end{equation*}
We perturb $E$ through the flow $\Phi$ by defining $E_t := \Phi(E,t)$ for $t\in(-\delta,\delta)$.  
The zero average condition on $\varphi$ guarantees that we may choose $X$ in  such a way that  the flow  preserves the volume 
up to a small error, i.e.,  $\gamma(E_t)= \gamma(E) + o(t^2)$. 
Then the \textbf{first variation} condition for the minimizer 
\begin{equation*}
\frac{\partial}{\partial t}\F(E_t)|_{t=0} = 0
\end{equation*}
leads to the Euler equation
\begin{equation*}
\Hh -\langle x, \nu \rangle + \e \langle b, x \rangle = \lambda \qquad \text{on }\, \partial E,
\end{equation*}
where $b = b(E)$ is the barycenter of $E$, $\nu=\nu^E$ the exterior normal of $\partial E$,
and $\lambda$ is the Lagrange multiplier. Furthermore, the \textbf{second variation} condition for the minimizer
\begin{equation*}
\frac{\partial^2}{\partial t^2}\F(E_t)|_{t=0} \geq 0
\end{equation*}
leads to 
\begin{equation}\label{second variation}
\int_{\partial E} \left(|D_\tau \varphi|^2 - |B_E|^2\varphi^2 - \varphi^2
+ \e \langle b , \nu \rangle \varphi^2\right)\, d\Ha^{n-1}_\gamma(x) 
+ \frac{\e}{\sqrt{2\pi}} \, \Big| \int_{\partial E} \varphi\, x \, d\Ha^{n-1}_\gamma(x) \Big|^2 \geq 0.
\end{equation}

\vspace{4pt}
In the second part we effectively use  the Euler
equation and the second variation condition  to prove  that half spaces are the unique minimizers of $\F$.
Given a minimizer $E$, assume (without loss of generality) that its barycenter is in  direction 
$- e^{(n)}$, i.e.,  $b(E) = -|b|e^{(n)}$.
As we said, we have to show that $E = H_{e^{n},s}$. In order to understand how the profile of the set $E$
varies in the directions perpendicular to $e^{n}$, the \emph{key idea} is to use as $\varphi$
 the functions $\nu_j$, $j\in\{1,\ldots,n-1\}$, where $\nu_j=\la \nu, e^{(j)} \ra$.
We are allowed to do this because $\nu_j$ has zero average (see \eqref{test has zero average}).
From the Euler equation we get 
\begin{equation*}
\Bigl| \int_{\partial E} \nu_j\, x  \, d\Ha^{n-1}_\gamma(x) \Bigl|^2 \leq C\int_{\partial E} \nu_j^2 \, d\Ha^{n-1}_\gamma(x)  
\end{equation*}
for some $C$ depending on $s$ (but not on $n$). Therefore, when $\e$ is small enough 
the second variation condition \eqref{second variation} provides the inequality
\begin{equation}\label{piece 1}
\int_{\partial E} \Bigl(|D_\tau \nu_j|^2 - |B_E|^2\nu_j^2  + \e  |b|  \nu_n \nu_j^2 - \frac{1}{2}\nu_j^2
\Bigr)\, d\Ha^{n-1}_\gamma(x) \geq 0. 
\end{equation}

Let  $\delta_j$ be the tangential derivative in $e^{(j)}$-direction and $\Delta_\tau $ the tangential Laplacian. 
By differentiating the Euler equation with respect to $\delta_j$ and by using the geometric equality
\begin{equation*}
\Delta_\tau \nu_j = -  |B_E|^2\nu_j + \delta_j \Hh \qquad \text{on }\, \partial E
\end{equation*}
we deduce 
\begin{equation*}
\Delta_\tau \nu_j - \la D_\tau \nu_j, x \ra= -  |B_E|^2\nu_j  - \e |b | \nu_n \nu_j  \qquad \text{on }\, \partial E.
\end{equation*}
We multiply the above equation by $\nu_j$, integrate it over $\partial E$ and use the divergence theorem on hypersurfaces to get
\begin{equation}\label{piece 2}
\int_{\partial E} \Bigl(|D_\tau \nu_j|^2 - |B_E|^2\nu_j^2  + \e  |b|  \nu_n \nu_j^2 \Bigr)\, d\Ha^{n-1}_\gamma(x) =0. 
\end{equation}
By comparing \eqref{piece 1} and \eqref{piece 2} we conclude that necessarily $\nu_j \equiv 0$ on $\partial E$, i.e., 
$E$ is constituted by strips perpendicular to $e^{n}$. To complete the proof  we show that  $\partial E$ is connected, 
which implies that  $E$ is the half-space $H_{e^{n},s}$.

%%%%%%%%%%%%%%%%%%%%%%%%%%%%%%%%%%%%%%%%%%%%%%%%%%%%%%%%%%%%%%%%%%%%%%%%%%%%%%%%%%%%%%%%%%%%%%%%%%%%%%%%%%%%%%%%%%%%%%%%%%%%
\section{Minimization problem}\label{sect min}

\vspace{8pt}

In this section we study  the functional  $\F:\B\to\R^+$ defined by
\begin{equation}\label{functional}
\mathcal{F} (E) =  P_\gamma(E)+ \e\sqrt{\pi/2}\,|b(E)|^2  + \Lambda\sqrt{2\pi}\, \big| \gamma(E) - \phi(s) \big|,
\end{equation}
where $\e>0$, $\Lambda>0$, and $s\leq0$ are given.  The last term is  a volume penalization that forces (for $\Lambda$ 
large enough) the minimizers of $\F$ to have Gaussian measure $\phi(s)$. We first prove the existence of minimizers and 
then study their regularity. We calculate also the Euler equation and the second variation
of $\mathcal{F}$. All these results are nowadays standard, but for the reader's convenience
we prefer to give  the proofs. Specific properties of the minimizers
will be analyzed in the next section, along the proof of our Main Theorem.

\begin{proposition}
The functional $\F$ has a minimizer.
\end{proposition}

\begin{proof}
Consider a sequence $E_h$ in $\B$ such that
\begin{equation*}
\lim_{h\rightarrow\infty}\F(E_h)=\inf\{\F(F) \colon F\in\B\}.
\end{equation*}
Since for any bounded open set $A\subset\R^n$ one has that $\sup_h P(E_h; A) $ is finite,
the compactness theorem for $BV$ functions (see \cite[Theorem 3.23]{AFP}) ensures the existence
of a Borel set $E\subset\R^n$ such that, up to a subsequence, 
$\chi_{E_h}\rightarrow\chi_E$ strongly in $L^1_{\loc}(\R^n)$. 
Given $R>0$, let $r_h$ and $r$ be such that
\begin{equation*}
\phi(r_h)=\gamma(E_h\setminus B_R)
\;\text{ and }\; \phi(r)=\gamma(\R^n\setminus B_R).
\end{equation*}
From inequality \eqref{b_s} we get
\begin{equation*}
\Bigl|\int_{E_h\setminus B_R}x\,d\gamma(x)\Bigr|
\leq\frac{e^{-\frac{r_h^2}{2}}}{\sqrt{2\pi}}
\leq\frac{e^{-\frac{r^2}{2}}}{\sqrt{2\pi}}.
\end{equation*}
A similar estimate holds also for the set $F\setminus B_R$. Therefore, since
\begin{equation*}
\Bigl|\int_{E_h}x\,d\gamma(x)-\int_E x\,d\gamma(x)\Bigr|
\leq \Bigl|\int_{\R^n}(\chi_{E_h}-\chi_E)\chi_{B_R} x\,d\gamma(x)\Bigr|+\frac{2e^{-\frac{r^2}{2}}}{\sqrt{2\pi}},
\end{equation*}
we have that $b(E)=\lim_{h\rightarrow\infty}b(E_h)$. 
Equation \eqref{g perimeter} implies that the Gaussian perimeter is lower
semicontinuous with respect to $L^1_{\loc}$ convergence of sets, namely
$P_\gamma(E)\leq\liminf_{h\rightarrow\infty}P_\gamma(E_h)$, so that
$\F(E)\leq \F(F)$ for every set $F\in\B$.
\end{proof}

\vspace{8pt}
The regularity of the  minimizers of  $\F$ follows from the regularity theory for  almost minimizers of the perimeter \cite{Ta}. 
From the regularity point of view the advantage of having the strong asymmetry in the functional \eqref{functional} instead of
the standard one is that the minimizers are smooth outside the singular set. The fact that one may gain regularity by 
replacing the standard asymmetry by a stronger one is also observed in a different context in \cite{BDV}.

\begin{proposition}\label{1st_variation}
Let $E$ be a minimizer of $\F$ defined in \eqref{functional}. Then the reduced boundary  $\partial^*E$ is a relatively open, 
smooth hypersurface and satisfies the Euler equation
\begin{equation}\label{euler}
\Hh -\langle x, \nu \rangle + \e \langle b, x \rangle = \lambda \qquad \text{on }\, \partial^*E,
\end{equation}
where $b=b(E)$ and $\nu=\nu^E$.
Here $\lambda$ is the Lagrange multiplier which can be estimated by
\begin{equation*}
|\lambda| \leq \Lambda.
\end{equation*}
The singular part of the boundary $\partial E \setminus \partial^*E$ is empty when $n <8$, 
while for $n \geq 8$ its Hausdorff dimension can be estimated by $\dim_{\Ha}(\partial E \setminus \partial^*E) \leq n-8$.  
\end{proposition}

\begin{proof}
First of all, we note  that $\partial E$ is the topological boundary of a properly chosen representative of the set
(see \cite[Proposition~12.19]{Ma}). 

Let us fix $x_0 \in \partial E$ and $r \in (0,1)$. From the minimality we deduce that for every set 
$F \subset \R^n$ with locally finite perimeter such that $F \triangle E \subset B_{2r}(x_0)$ it holds 
\begin{equation} \label{almost_min1}
P_\gamma(E) \leq P_\gamma(F) + C\gamma(F \triangle E )
\end {equation}
for some constant $C$ depending on $|x_0|$. If we choose $F = E \cup B_r(x_0)$ we get from \eqref{almost_min1} that
\begin{equation*}
P_\gamma(E) \leq P_\gamma(E  \cup   B_r(x_0)) + C\gamma(B_r(x_0)).
\end{equation*}
On the other hand, arguing  as in  \cite[Lemma 12.22]{Ma} we obtain 
\begin{equation*}
P_\gamma(E \cup B_r(x_0)) + P_\gamma(E \cap B_r(x_0)) \leq P_\gamma(E) + P_\gamma(B_r(x_0)).
\end{equation*}
The previous two inequalities yield
\begin{equation*}
P_\gamma(E \cap B_r(x_0))  \leq  P_\gamma(B_r(x_0)) + C\gamma(B_r(x_0)) \leq Cr^{n-1}.
\end{equation*}
The left-hand side can be estimated simply by 
\begin{equation*}
P_\gamma(E \cap B_r(x_0))  \geq ce^{-|x_0|^2} P(E; B_r(x_0)).
\end{equation*}
Therefore we obtain
\begin{equation} \label{upper_est}
P(E; B_r(x_0)) \leq C_0r^{n-1}
\end {equation}
for some constant $C_0 = C_0(|x_0|)$.  Note that for every $x \in B_r(x_0)$ and $r \in (0,1)$ it holds
\begin{equation*}
 \big| e^{-\frac{|x|^2}{2}} - e^{-\frac{|x_0|^2}{2}} \big|  \leq  Cr 
\end{equation*}
for some constant $C$. Therefore,  \eqref{almost_min1} and \eqref{upper_est} imply that for all sets $F$ with 
$F \triangle E \subset \! \subset B_r(x_0)$ and $r \leq 1$ it holds
\begin{equation*}
P(E; B_r(x_0)) \leq P(F; B_r(x_0)) + Cr^n
\end {equation*}
for some constant $C$ depending on $|x_0|$. It follows from \cite[Theorem 1.9]{Ta} (see also \cite[Theorem 21.8]{Ma}) 
that $\partial^*E$ is a relatively open (in $\partial E$) $C^{1,\sigma}$ hypersurface for every $\sigma < 1/2$,
and that the singular set $\partial E \setminus \partial^*E$ is empty when $n < 8$,
while $\dim_{\Ha}(\partial E \setminus \partial^*E) \leq n-8$ when $n \geq 8$. 

Let us next prove that $\partial^* E$ satisfies the Euler equation \eqref{euler}. Since $\partial^* E$ is relatively open we  
find an open set $U \subset \R^n$ such that $\partial E \cap U = \partial^* E$. 
Let us first prove that  for every $X \in C_0^1(U;\R^n)$
with $\int_{\partial^* E}  \langle X, \nu \rangle \, d\Ha^{n-1}_\gamma(x) = 0$ we have
\begin{equation} \label{weak euler 1}
\int_{\partial^* E} \div_\tau X - \langle X, x \rangle\, d\Ha^{n-1}_\gamma(x)
+\e \int_{\partial^* E} \langle b, x \rangle  \langle X, \nu \rangle \, d\Ha^{n-1}_\gamma(x) = 0.
\end{equation}
To this aim let $\Phi : U\times(-\delta,\delta) \to U$ be the flow associated with $X$, i.e., 
\begin{equation*}
\frac{\partial}{\partial t} \Phi(x,t) = X( \Phi(x,t) ), \qquad \Phi(x,0) = x.
\end{equation*}
There exists a time  interval $(-\delta,\delta)$  such that the flow $\Phi$ is defined
in $U\times(-\delta,\delta)$,  it  is $C^1$ regular  and for every $t \in (-\delta,\delta)$  the map $x \mapsto \Phi(x,t)$ 
is a local $C^1$ diffeomorphism  \cite[Theorem 6.1]{Te}. Because $X$ vanishes near the boundary of $U$,   
$\Phi(x,t) = x$ for every point $x$ near $\partial U$.  With this in mind we extend the flow to every 
$t \in (-\delta,\delta)$ and $x \in  \R^n \setminus U$   by $\Phi(x,t) = x$. Then for small values of  $t$ 
the map $x \mapsto \Phi(x,t)$ is a $C^1$ diffeomorphism. We define $E_t := \Phi(E,t)$.
Let us denote  the Jacobian of $\Phi(\cdot, t)$ by $J\Phi(x,t)$ and the  
tangential Jacobian on $\partial^* E$  by $J_\tau \Phi(x,t)$. We recall the formulas (see \cite{StZ})
\begin{equation}\label{taylor volume}
\frac{\partial}{\partial t} \big|_{t=0}  J\Phi(x,t) = \div X \qquad \text{and} \qquad \frac{\partial}{\partial t} \big|_{t=0}J_\tau\Phi(\cdot,t)  = \div_\tau X.
\end{equation}
Note also that by definition $\frac{\partial}{\partial t} \big|_{t=0} \Phi(x,t) = X(x)$ and
$\Phi(x,0)= x$. Then we have by change of variables
\begin{equation}
\label{triviaali}
\begin{split}
\frac{\partial}{\partial t} \big|_{t=0}  \gamma (E_t)  
&=  \frac{\partial}{\partial t} \big|_{t=0} \left( \frac{1}{(2\pi)^{\frac{n}{2}}}\int_{E} e^{-\frac{|\Phi(x,t)|^2}{2}} J \Phi(x,t) \, dx \right)\\
&=  \frac{1}{(2\pi)^{\frac{n}{2}}}\int_{E} ( \div X  - \la X, x\ra)e^{-\frac{|x|^2}{2}} \, dx  \\
&=  \frac{1}{(2\pi)^{\frac{n}{2}}}\int_{E} \div (e^{-\frac{|x|^2}{2}}X) \, dx\\
&=  \frac{1}{\sqrt{2\pi}}\int_{\partial^* E}  \langle X, \nu \rangle \, d\Ha^{n-1}_\gamma(x) = 0.
\end{split}
\end{equation}
This means that  $X$ produces a zero first order volume variation of $E$ and, therefore,
\begin{equation*}
\frac{\partial}{\partial t}|_{t=0} \bigl| \gamma(E_t) - \phi(s)  \bigl| =0.
\end{equation*}
We obtain the  formula \eqref{weak euler 1} by  the minimality of $E$ and by change of variables
\begin{equation*}
\begin{split}
\frac{\partial}{\partial t} \big|_{t=0}   P_\gamma(E_t)  &=  \frac{\partial}{\partial t} \big|_{t=0} 
\left( \int_{\partial^* E} \Bigl( e^{-\frac{|\Phi(x,t)|^2}{2}} J_\tau \Phi(x,t) \Bigr) \, d\Ha^{n-1}(x)\right)  \\
&=\int_{\partial^* E} \div_\tau X - \langle X, x \rangle\, d\Ha^{n-1}_\gamma(x)
\end{split}
\end{equation*}
and 
\begin{equation*}
\begin{split}
\frac{\partial}{\partial t} \big|_{t=0}   | b (E_t)|^2   
&=  \frac{1}{(2\pi)^n}\,\frac{\partial}{\partial t} \big|_{t=0} \Bigl| \int_{E} \Bigl(\Phi(x,t) \, e^{-\frac{|\Phi(x,t)|^2}{2}} J \Phi(x,t) \Bigr) \, dx \Bigl|^2 \\
&= \frac{2}{(2\pi)^{\frac{n}{2}}} \int_{E} \bigl( \la b, X \ra - \la b, x\ra \la X, x\ra+  \la b, x \ra \div X \bigr)e^{-\frac{|x|^2}{2}} \, dx  \\
&= \frac{2}{(2\pi)^{\frac{n}{2}}}  \int_{E} \div \Bigl( \la b, x \ra  e^{-\frac{|x|^2}{2}} \, X \Bigr) dx \\
&= \frac{2}{\sqrt{2\pi}}\int_{\partial^* E } \langle b, x \rangle  \langle X, \nu \rangle  \, d\Ha^{n-1}_\gamma(x). 
\end{split}
\end{equation*}

We use \eqref{weak euler 1} to show that $\partial^* E$ satisfies the Euler equation \eqref{euler} in a weak sense, i.e., 
there exists a number $\lambda \in \R$ such that  for every $X \in C_0^1(U;\R^n)$  we have
\begin{equation} \label{weak euler 2}
\int_{\partial^* E} \div_\tau X - \langle X, x \rangle\, d\Ha^{n-1}_\gamma(x)
+\e \int_{\partial^* E} \langle b, x \rangle  \langle X, \nu \rangle \, d\Ha^{n-1}_\gamma(x) = \lambda  \int_{\partial^* E}\langle X, \nu \rangle \, d\Ha^{n-1}_\gamma(x).
\end{equation}
Let $X_1, X_2   \in C_0^1(U;\R^n)$ be such that $\int_{\partial^* E}  \langle X_i, \nu \rangle \, d\Ha^{n-1}_\gamma(x) \neq 0$, $i =1,2$. Denote $\alpha_1 = \int_{\partial^* E}  \langle X_1, \nu \rangle \, d\Ha^{n-1}_\gamma(x)$ and $\alpha_2 = \int_{\partial^* E}  \langle X_2, \nu \rangle \, d\Ha^{n-1}_\gamma(x)$, and define 
\begin{equation*}
X:= X_1 - \frac{\alpha_1}{\alpha_2} X_2.
\end{equation*}
Then $X \in C_0^1(U;\R^n)$ satisfies  $\int_{\partial^* E}  \langle X, \nu \rangle \, d\Ha^{n-1}_\gamma(x) = 0$ 
and \eqref{weak euler 1} implies
\begin{equation*}
\begin{split}
&\frac{1}{\alpha_1}\left(  \int_{\partial^* E}  \div_\tau X_1 - \langle X_1, x \rangle\, d\Ha^{n-1}_\gamma(x)
+\e \int_{\partial^* E} \langle b, x \rangle  \langle X_1, \nu \rangle \, d\Ha^{n-1}_\gamma(x) \right)\\
&= \frac{1}{\alpha_2}\left(  \int_{\partial^* E}  \div_\tau X_2- \langle X_2, x \rangle\, d\Ha^{n-1}_\gamma(x)
+\e \int_{\partial^* E} \langle b, x \rangle  \langle X_2, \nu \rangle \, d\Ha^{n-1}_\gamma(x) \right).
\end{split}
\end{equation*}
Therefore there exists $\lambda \in \R$ such that \eqref{weak euler 2} holds.

Since the reduced boundary $\partial^*E$ is a $C^{1,\sigma}$ manifold and since  it satisfies the Euler equation \eqref{euler} 
in a weak sense, from classical Schauder estimates we deduce that $\partial^*E$ is in fact a $C^{\infty}$ hypersurface. 
In particular, we conclude  that the Euler equation  \eqref{euler} holds pointwise on $\partial^*E$. 

Finally, in order to bound the Lagrange multiplier $\lambda$,  let  $X \in C_0^1(U; \R^n)$  
be any vector field, and let $\Phi(x,t)$,  $E_t = \Phi(E,t)$ be as above. 
Then by the above calculations we have
\begin{equation*}\begin{split}
\frac{\partial}{\partial t} \big|_{t=0}  \left( P_\gamma(E_t) 
+\e\sqrt{\pi/2}\, |b(E_t)|^2 \right) &= \int_{\partial^* E} \div_\tau X - \langle X, x \rangle
+\e \langle b, x \rangle  \langle X, \nu \rangle  \, d\Ha^{n-1}_\gamma(x)\\
&=\int_{\partial^* E}  (\Hh -\langle x, \nu \rangle + \e \langle b, x \rangle)  \langle X, \nu \rangle  \, d\Ha^{n-1}_\gamma(x)\\
&= \lambda \int_{\partial^* E}   \langle X, \nu \rangle  \, d\Ha^{n-1}_\gamma(x)
\end{split}\end{equation*}
and 
\begin{equation*}
\limsup_{t \to 0} \frac{\big| \gamma(E_t) - \phi(s) \big| - \big| \gamma(E) - \phi(s) \big|}{t} 
\leq \Big| \frac{\partial}{\partial t} \big|_{t=0}   \gamma(E_t) \Big|   
=\frac{1}{\sqrt{2\pi}}\,\big|  \int_{\partial^* E}  \langle X, \nu \rangle \, d\Ha^{n-1}_\gamma(x) \big| .
\end{equation*}
Therefore by the minimality of $E$ we have
\begin{equation*}
\lambda \int_{\partial^* E}   \langle X, \nu \rangle  \, d\Ha^{n-1}_\gamma(x)
+\Lambda  \big|  \int_{\partial^* E}  \langle X, \nu \rangle \, d\Ha^{n-1}_\gamma(x) \big| \geq 0
\end{equation*}
for every $X \in C_0^1(U; \R^n)$. This proves the claim. 
\end{proof}

\vspace{8pt}
Next  we derive the second order condition for minimizers of the functional $\F$, i.e.,  the quadratic form 
associated with the second variation is non-negative. 
Let us briefly explain what we mean by this. Let  $\varphi : \partial^* E \to \R$ be a  smooth function with 
compact support such that it has zero average, i.e.,   $\int_{\partial^* E}\varphi \, d\Ha^{n-1}_\gamma(x)= 0$.   
We choose a specific vector field $X: \R^n \to \R^n$, such that  $X:= \varphi \nu^E$ on $\partial^* E$. 
We denote the associated flow by $\Phi$ and define $E_t := \Phi(E,t)$.  
We note that  since $\varphi$ has zero average then by \eqref{triviaali} $X$ produces a zero
first order volume variation of $E$. This enables us to define $X$ in  such a way that the volume variation 
produced by $X$ is zero up to second order, i.e., $\gamma(E_t) = \gamma(E) + o(t^2)$ (see \eqref{2nd_volume}
and \eqref{zero up to 2nd}). Therefore, under the condition that $\varphi$ has zero average, 
the volume penalization term in the functional $\F$ is negligible.
The second variation of the functional $\F$ at $E$ in the direction $\varphi$ is then defined to be the value
\[
\frac{d^2}{dt^2}\big|_{t=0}  \F(E_t) .
\]
It turns out that the choice of the vector field $X$ ensures that the second derivative exists and it 
follows from the minimality of $E$ that this value is non-negative.  
Moreover, the second variation at $E$ defines a quadratic form over all functions  
$\varphi \in C_0^\infty(\partial^* E)$ with zero average. 

The calculations of the second variation are standard (see \cite{AFM, Ma, MR, StZ} for similar cases) but since they 
are technically challenging we include them for the reader's convenience. We note that since $\partial E$ is not necessarily 
smooth we may only perturb the regular part of the boundary. We write $u \in C_0^\infty(\partial^* E)$ 
when $u: \partial^* E \to \R$ is a smooth function with compact support.

\begin{proposition}\label{2nd_variation}
Let $E$ be a  minimizer of  $\F$. The quadratic form associated with the second variation is non-negative
\begin{equation*}
J[\varphi] := \int_{\partial^* E} \left(|D_\tau \varphi|^2 - |B_E|^2\varphi^2 - \varphi^2
+ \e \langle b , \nu \rangle \varphi^2\right)\, d\Ha^{n-1}_\gamma(x) 
+ \frac{\e}{\sqrt{2\pi}} \, \Big| \int_{\partial^* E} \varphi\, x \, d\Ha^{n-1}_\gamma(x) \Big|^2 \geq 0
\end{equation*}
for every $\varphi \in C_0^\infty(\partial^* E)$ which satisfies 
\begin{equation*}
\int_{\partial^* E} \varphi \, d\Ha^{n-1}_\gamma(x) = 0.
\end{equation*}
Here $b = b(E)$ and $\nu=\nu^E$, while $|B_E|^2$ is the sum of the squares of the curvatures.
\end{proposition}

\begin{proof}
Assume that $\varphi \in C_0^\infty(\partial^* E)$  satisfies $\int_{\partial^* E} \varphi \, d\Ha^{n-1}_\gamma(x) = 0$. 
Let $d_E:\R^n \to \R$ be the signed distance function of $E$
\begin{equation*}
d_E(x) := \begin{cases}
&\text{dist}(x, \partial E)    \quad \text{for  }\,x \in \R^n \setminus  E \\
&-\text{dist}(x,\partial E)    \quad \text{for  }\,x \in E.
\end{cases}
\end{equation*}
It follows from Proposition \ref{1st_variation} that there is an open set $U \subset \R^n$ such that $d_E$ is 
smooth in $U$ and the support of $\varphi$ is in $\partial^* E \cap U$. We extend $\varphi$ to $U$, and call the extension simply 
 by $\varphi$, so that $\varphi \in C_0^\infty(U)$ and 
\begin{equation} \label{2nd_volume}
\partial_\nu \varphi = ( \la x, \nu \ra- \Hh )\varphi \qquad \text{on }\, \partial^* E.
\end{equation}
Finally, we define the vector field $X: \R^n \to \R^n$ by  $X := \varphi \nabla d_E$ in $U$ and $X:= 0$ in $\R^n \setminus U$. 
Note that $X$ is smooth  and $X=  \varphi \nu$ on $\partial^* E$. 

Let $\Phi : \R^n \times (-\delta , \delta) \to\R^n$ be the flow associated with  $X$, i.e., 
\begin{equation*}
\frac{\partial}{\partial t} \Phi(x,t) = X( \Phi(x,t) ), \qquad \Phi(x,0) = x
\end{equation*} 
and define $E_t = \Phi(E,t)$. Let us denote  the Jacobian of $\Phi(\cdot, t)$ by $J\Phi(x,t)$ and the  
tangential Jacobian on $\partial^* E$  by $J_\tau \Phi(x,t)$. We recall the formulas \eqref{taylor volume}
and also (see again \cite{StZ}) the formulas
\begin{equation} \label{taylor perimeter}
\begin{split}
&\frac{\partial^2}{\partial t^2}\big|_{t = 0} J\Phi(x,t)  = \div((\div X)X)\\
&\frac{\partial^2}{\partial t^2}\big|_{t = 0} J_\tau\Phi(\cdot,t) 
=  |(D_\tau X)^T\nu|^2  + (\div_\tau X)^2  + \div_\tau Z - \text{Tr}(D_\tau X)^2
\end{split}
\end{equation}
where $Z:=\frac{\partial^2 \Phi(x,t)}{\partial t^2}\big|_{t = 0}$ is the acceleration field. Recall also that 
by definition $\Phi(x,0)= x$ and $\frac{\partial}{\partial t} \big|_{t=0}\Phi(x,t)= X$.

We begin by differentiating the Gaussian volume. Similarly to \eqref{triviaali}, by a change of variables 
we use \eqref{taylor volume} and \eqref{taylor perimeter} to calculate
\begin{equation*}
\frac{\partial}{\partial t} \big|_{t=0}  \gamma (E_t) = \frac{1}{\sqrt{2\pi}}\int_{\partial^* E}  \varphi \, d\Ha^{n-1}_\gamma(x) = 0
\end{equation*}
and  
\begin{equation} \label{zero up to 2nd}
\begin{split}
\frac{\partial^2}{\partial t^2}\big|_{t = 0}  \gamma (E_t) 
&= \frac{1}{(2\pi)^\frac{n}{2}}\int_E \div \Bigl( \div( X e^{-\frac{|x|^2}{2}}) X \Bigr) \, dx \\
&= \frac{1}{\sqrt{2\pi}}\int_{\partial^* E} \varphi\partial_\nu \varphi + (\Hh - \la x, \nu \ra)\varphi^2 \, d\Ha^{n-1}_\gamma(x) = 0,
\end{split}
\end{equation}
where the last equality comes from \eqref{2nd_volume}.  Hence, $\gamma(E_t) =\gamma(E)+ o(t^2)$ and 
\begin{equation*}
\frac{\partial^2}{\partial t^2}\big|_{t = 0} \big| \gamma(E_t) -\phi(s) \big| = 0.
\end{equation*}
Since $t  \mapsto P_\gamma(E_t)$ and $t  \mapsto |b(E_t)|^2$ are smooth with respect to $t$ we have by the minimality of $E$ that 
\begin{equation} \label{volume_const}
0 \leq  \frac{\partial^2}{\partial t^2}\big|_{t = 0}  \mathcal{F} (E_t) 
=  \frac{\partial^2}{\partial t^2}\big|_{t = 0}P_\gamma (E_t) +\e\sqrt{\pi/2}\,\frac{\partial^2}{\partial t^2}\big|_{t = 0} |b(E_t)|^2.
\end{equation}
Thus we need to differentiate the perimeter and the barycenter. 

To differentiate the perimeter we write
\begin{equation*}
P_\gamma (E_t)  = \int_{\partial^* E} e^{-\frac{|\Phi(x,t)|^2}{2}}\, J_\tau\Phi(x,t) \, d\Ha^{n-1}(x).
\end{equation*}
We differentiate this twice and  use \eqref{taylor perimeter} to get 
\begin{equation}\label{der_perimeter}
\begin{split}
\frac{\partial^2}{\partial t^2}\big|_{t = 0}  P_\gamma (E_t) =& \int_{\partial^* E} \left( |(D_\tau X)^T\nu|^2 
+ (\div_\tau X)^2  + \div_\tau Z - \text{Tr}(D_\tau X)^2 \right) \, d\Ha^{n-1}_\gamma(x)\\
&+\int_{\partial^* E} \left(-2 \div_\tau X \la X,x \ra - \la Z,x \ra - |X|^2 + \la X, x \ra^2\right) \, d\Ha^{n-1}_\gamma(x)\\
=&\int_{\partial^* E} \bigl(|D_\tau \varphi|^2 -|B_E|^2 \varphi^2 -\varphi^2\bigr)  \, d\Ha^{n-1}_\gamma(x)\\
&+\int_{\partial^* E} (\Hh - \la x, \nu \ra ) (\varphi \partial_\nu \varphi
+(\Hh - \la x, \nu \ra )\varphi^2) \, d\Ha^{n-1}_\gamma(x).
\end{split}
\end{equation}

Let us denote $b_t = b(E_t)$, $\dot{b}=\frac{\partial}{\partial t}\big|_{t = 0}  b_t$ 
and $\ddot{b}=\frac{\partial^2}{\partial t^2}\big|_{t = 0}  b_t$. Then 
\begin{equation*}
\frac{\partial^2}{\partial t^2}\big|_{t = 0}  |b_t|^2 = 2 \la b, \ddot{b} \ra  + 2|\dot{b}|^2. 
\end{equation*}
To differentiate the barycenter we write
\begin{equation*}
b_t   =\frac{1}{(2\pi)^\frac{n}{2}}\int_{E} \Phi(x,t) \, e^{-\frac{|\Phi(x,t)|^2}{2}}\, J\Phi(x,t) \, dx.
\end{equation*}
We use \eqref{taylor volume} and \eqref{taylor perimeter}, and get after differentiating once that 
\begin{equation}\label{der_barycenter_1}
\dot{b} =\frac{1}{\sqrt{2\pi}}\int_{\partial^* E} \varphi\, x  \, d\Ha^{n-1}_\gamma(x)
\end{equation}
and after differentiating twice that
\begin{equation*}
\begin{split}
\ddot{b} &= \frac{1}{(2\pi)^\frac{n}{2}}\int_E    \bigl( x\,  \div \left( (\div X) X\right)  
+ 2X (\div X) - 2x \la X, x \ra (\div X) - 2X \la X, x \ra \bigr) \, e^{-\frac{|x|^2}{2}} \, dx  \\
&\,\,\,\,\,\,+  \int_E    \left( (DX)X + x\la X, x \ra^2 - x\la DX X, x \ra - x |X|^2   \right) \, e^{-\frac{|x|^2}{2}} \, dx\\
&= \frac{1}{(2\pi)^\frac{n}{2}}\int_E \Bigl( (DX) X e^{-\frac{|x|^2}{2}}  + 2 X\,  \div (X  e^{-\frac{|x|^2}{2}})  
+ x \, \div \Bigl(\div (X  e^{-\frac{|x|^2}{2}}) X \Bigr)  \Bigr) \, dx.
\end{split}
\end{equation*}
Thus we obtain by the divergence theorem that 
\begin{equation}\label{der_barycenter_2}
\begin{split}
\la b, \ddot{b} \ra  &= \frac{1}{(2\pi)^\frac{n}{2}}\int_{ E}\div  \Bigl( \la X, b \ra  X e^{-\frac{|x|^2}{2}}  \Bigr)   
+\div  \Bigl( \la x, b \ra  \Bigl(\div (X  e^{-\frac{|x|^2}{2}}) X \Bigr) \Bigr) \, dx\\
&=\frac{1}{\sqrt{2\pi}}\int_{\partial^* E}  \la X, b \ra  \la X, \nu \ra\, d \Ha^{n-1}_\gamma(x) 
+\frac{1}{\sqrt{2\pi}}\int_{\partial^* E}  \la b, x  \ra \la X, \nu \ra \Bigl(\div (X  e^{-\frac{|x|^2}{2}}) \Bigr) \, d\Ha^{n-1}(x) \\
&=\frac{1}{\sqrt{2\pi}}\int_{\partial^* E}  \la b, \nu \ra  \varphi^2 \, d\Ha^{n-1}_\gamma(x) 
+\frac{1}{\sqrt{2\pi}}\int_{\partial^* E}  \la b, x \ra ( \varphi\partial_\nu \varphi +(\Hh - \la x, \nu \ra)\varphi^2) \, d\Ha^{n-1}_\gamma(x).
\end{split}
\end{equation}

Therefore  \eqref{volume_const}, \eqref{der_perimeter}, \eqref{der_barycenter_1} and \eqref{der_barycenter_2} imply
\begin{equation}\label{long_calculation}
\begin{split}
0 \leq&  \frac{\partial^2}{\partial t^2}\big|_{t = 0}  \mathcal{F} (E_t) 
=  \frac{\partial^2}{\partial t^2}\big|_{t = 0}    P_\gamma (E_t) +  \e\sqrt{2\pi} \left(  \la b, \ddot{b}  \ra  +  |\dot{b} |^2 \right) \\
=& \int_{\partial^* E}\left( |D_\tau \varphi|^2 -|B_E|^2 \varphi^2 -\varphi^2  
+ \e \la b, \nu \ra  \varphi^2   \right)  \, d\Ha^{n-1}_\gamma(x) 
+\frac{\e}{\sqrt{2\pi}} \, \Big| \int_{\partial^* E}\varphi\, x \, d\Ha^{n-1}_\gamma(x) \Big|^2\\
&+\frac{1}{\sqrt{2\pi}}\int_{\partial^* E} (\Hh - \la x, \nu \ra + \e \la b, x \ra ) (\varphi \partial_\nu \varphi+  (\Hh 
-\la x, \nu \ra )\varphi^2) \, d\Ha^{n-1}_\gamma(x).
\end{split}
\end{equation}
We use the Euler equation \eqref{euler} and  \eqref{2nd_volume}  to conclude that 
\begin{equation*}
\begin{split}
\int_{\partial^* E} &(\Hh - \la x, \nu \ra + \e \la b, x \ra ) (\varphi \partial_\nu \varphi
+(\Hh-\la x, \nu \ra )\varphi^2) \, d\Ha^{n-1}_\gamma(x)\\
&=\lambda \int_{\partial^* E} \varphi \partial_\nu \varphi+  (\Hh - \la x, \nu \ra )\varphi^2 \, d\Ha^{n-1}_\gamma(x) = 0.
\end{split}
\end{equation*}
Hence, the claim follows from \eqref{long_calculation}.
\end{proof}

\vspace{8pt}
We would like to extend the quadratic form in Proposition \ref{2nd_variation} to more general functions 
than $\varphi \in C_0^\infty(\partial^* E)$. To this aim we define the function space $H_\gamma^1(\partial^*E)$ 
as the closure of $C_0^\infty(\partial^*E)$ with respect to the  norm
$||u||_{H_\gamma^1(\partial^*E)} = ||u||_{L_\gamma^2(\partial^*E)} + ||D_\tau u||_{L_\gamma^2(\partial^*E,\R^n)}$. 
Here $L_\gamma^2(\partial^*E)$ is the set of square integrable functions on $\partial^*E$ with respect to
the measure $\gamma$.
A priori the definition of $H_\gamma^1(\partial^*E)$ seems rather restrictive since it is not clear if  even 
constant functions belong to $H_\gamma^1(\partial^*E)$.  However, the information on the singular set  
$\dim_{\Ha}(\partial E \setminus \partial^*E) \leq n-8$ from Proposition \ref{1st_variation}  ensures 
that the singular set  has capacity zero and it is therefore negligible. It follows that  every smooth function 
$u \in  C^\infty(\partial^* E)$ which has finite  $H_\gamma^1$-norm is in  $H_\gamma^1(\partial^*E)$. 
Recall that $\partial^*E$ is a relatively open, $C^\infty$ hypersurface. 
In particular,  if $u : \R^n \to \R$ is a smooth function such that  the   $H_\gamma^1(\partial^*E)$ norm of its restriction 
on $\partial^* E$  is bounded, then the restriction is in  $H_\gamma^1(\partial^*E)$. 

\begin{lemma}\label{sing_cap}
 Let $E$ be a minimizer of $\F$. If $u \in  C^\infty(\partial^* E)$ is such that 
$||u||_{H_\gamma^1(\partial^*E)} <\infty$, then $u \in H_\gamma^1(\partial^*E)$. 
\end{lemma}

\begin{proof}
By truncation we may assume that $u$ is bounded and by a standard mollification argument it is enough to 
find Lipschitz continuous functions $u_k$ with a compact support on $\partial^* E$  such that  
$\lim_{k \to \infty}||u-u_k||_{H_\gamma^1(\partial^*E)} =0$. We will show that there exist  Lipschitz continuous 
functions $\zeta_k: \partial^* E \to \R$ with  compact support such that 
$0\leq \zeta_k \leq 1$, $\zeta_k \to 1$ in $H_\gamma^1(\partial^* E)$ and $\zeta_k(x) \to 1$ pointwise on $\partial^*E$. 
We may then choose $u_k = u \zeta_k$ and the claim follows.

Let us fix $k \in \mathbb{N}$. First of all let us choose a large radius  $R_k$ such that the Gaussian perimeter of $E$ 
outside the ball $B_{R_k}$ is small, i.e., $P_\gamma(E; \R^n \setminus B_{R_k}) \leq 1/k$. We choose a cut-off function 
$\eta_k \in C_0^\infty(B_{2R_k})$ such that $|D\eta_k (x)|\leq 1$ for every $x \in \R^n$ and $\zeta \equiv 1$ in $B_{R_k}$. 

Denote  the singular set by $\Sigma := \partial E \setminus \partial^* E$. Proposition \ref{1st_variation} implies 
that $\Sigma$ is a closed set with    $\Ha^{n-3}(\Sigma ) = 0$. Therefore we may cover 
$\Sigma \cap \overline{B}_{2R_k}$ with  balls $B_{r_i} := B_{r_i}(x_i)$, $i = 1, \dots, N_k$, with radii $r_i \leq 1/2$ 
such that 
\begin{equation*}
\sum_{i=1}^{N_k} r_i^{n-3} \leq  \frac{1}{C_0}  \frac{1}{k}
\end{equation*}
where $C_0 = C_0(2R_k)$ is the constant from the estimate \eqref{upper_est} for the radius $2R_k$. For every ball $B_{2r_i}$ 
we define a cut-off function $\psi_i \in C_0^\infty(B_{2r_i})$ such that $\psi_i  \equiv 1$ in $B_{r_i}$, 
$0\leq \psi_i  \leq 1$ and $|D\psi_i|\leq \frac{2}{r_i} $. Define
\begin{equation*}
\theta_k(x):= \max_i  \psi_i(x), \qquad x \in \R^n. 
\end{equation*} 
Then $\theta_k(x) = 1 $ for   $x \in \cup_i B_{r_i}$,  $\theta_k(x) = 0 $ for $x \neq  \cup_i B_{2r_i}$ and 
it is Lipschitz continuous. We may  estimate its weak tangential gradient on $\partial^* E$  by 
\begin{equation*}
|D_\tau \theta_k(x)| \leq \max_{i} |D_\tau \psi_i(x)|  \leq \left(\sum_{i=1}^{N_k} |D\psi_i(x)|^2\right)^{1/2}
\end{equation*} 
for $\Ha^{n-1}$-almost every $x \in \partial^* E$.  Since $\Sigma \cap B_{2R_k}  \subset   \cup_i B_{r_i} $ the function
\begin{equation*}
\zeta_k = (1- \theta_k) \eta_k
\end{equation*} 
 has compact support on   $\partial^* E$. Note that by \eqref{upper_est} it holds $P(E; B_{2r_i}) \leq C_0  r_i^{n-1}$. 
Hence we have that
\begin{equation*}
\begin{split}
||D_\tau\zeta_k ||_{L_\gamma^2(\partial^*E)}^2 &\leq 2 \int_{\partial^*E} \left(|D_{\tau} \eta_k|^2 + |D_{\tau} \theta_k|^2\right) \, d\gamma(x)\\
&\leq 2 P_\gamma(E; \R^n \setminus B_{R_k})  + 2 \sum_{i=1}^{N_k} \int_{\partial ^* E \cap B_{2r_i}} |D\psi_i|^2\, d\Ha^{n-1}\\
&\leq \frac{2}{k} + 8 \sum_{i=1}^{N_k}r_i^{-2} P(E; B_{2r_i}) \\
&\leq  \frac{2}{k} + 8C_0  \sum_{i=1}^{N_k}r_i^{n-3} \leq \frac{10}{k}.
\end{split}
\end{equation*}
Similarly we conclude  that $||\zeta_k -1||_{L_\gamma^2(\partial^*E)}^2 \to 0$ as $k \to \infty$.
\end{proof}

%%%%%%%%%%%%%%%%%%%%%%%%%%%%%%%%%%%%%%%%%%%%%%%%%%%%%%%%%%%%%%%%%%%%%%%%%%%%%%%%%%%%%%%%%%%%%%%%%%%%%%%%%%%%%%%%%%%%%%%%%%%%
\section{Quantitative estimates}\label{quanti est}

\noindent
In this section we focus on the proof of our main result, as well as on some of its direct consequences. 
The proof of the Main Theorem is divided in several steps. The core of the  proof is step 3 where we
prove that any minimizer of the functional $\F$ is a half-space. In the final part of the proof (step 4) 
we only need to prove that every minimizer has the right volume.

\vspace{4pt}
\begin{proof}[Proof of the Main Theorem]
Since $\beta(E)=\beta(\R^n\setminus E)$, we may restrict ourselves to the case $s\leq0$.
As explained in Section \ref{idea}, we have to prove that the for some $\e$ and $\Lambda$ (only depending on $s$) 
the only minimizers of the functional $\F$ are the half-spaces~$H_{\unit,s}$, $\unit\in\Sf^{n-1}$.
We will show that this is indeed the case when we choose $\e$ and $\Lambda$ as 
\begin{equation}\label{epsilon}
\e = \frac{\sqrt{2\pi}e^{\frac{s^2}{2}}}{6(1+\Lambda^2)}  
\quad\text{ and }\quad
\Lambda = \frac{e^{-\frac{s^2}{2}}}{\sqrt{\pi}\,\phi(s)}.
\end{equation}
By plugging in this value of $\e$ in \eqref{here we are}, we have
\begin{equation}\label{here we are2}
\beta(E)\leq \frac{12}{\sqrt{2\pi}}(\Lambda^2+1) D(E).
\end{equation}
On the other hand, by second order analysis it is  easy to check that the function
\begin{equation*}
g(s):=e^{-\frac{s^2}{2}}+(\sqrt{2\pi}s-\pi)\phi(s)
\end{equation*}
is non-positive in $(-\infty,0]$. Indeed, $g'$ is non-positive and $\lim_{s\rightarrow-\infty}g(s)=0$.
Therefore, 
\begin{equation}\label{s vs lambda}
\Lambda^2  +1 =\frac{e^{-s^2}}{\pi\phi(s)^2} +1 \leq (\sqrt{\pi}-\sqrt{2}s)^2 +1  \leq 2\pi(1+s^2).
\end{equation}
By plugging in this estimate for $\Lambda$ in \eqref{here we are2},
we have \eqref{largo gente} with the  constant $c=12\sqrt{2\pi}$.

\vspace{8pt}
Assume now that $E$ is a minimizer of $\F$ and, without loss of generality,
that its barycenter is in the direction of $- e^{(n)}$, i.e.,  $b(E) = -|b|e^{(n)}$.
We will denote  $H_s=H_{e^{n},s}$ and show that $E = H_s$.  We divide the proof into four steps.

\vspace{4pt} 
\noindent \textbf{Step 1}.
As a first step we prove an upper bound for the quantity $\int_{\partial^* E} \la x, \unit \ra^2 \, d\gamma(x)$, i.e.,
for every $\unit \in \Sf^{n-1}$ it holds
\begin{equation*}
 \int_{\partial^* E}\la x, \unit \ra^2\, d\Ha^{n-1}_\gamma(x) \leq \frac{13}{3}(\Lambda^2+1)e^{-\frac{s^2}{2}}.
\end{equation*}
The proof is similar to the classical Caccioppoli inequality in the theory of elliptic equations.

We begin with few observations. Using $H_s$ as a competitor, the minimality of $E$ implies 
\begin{equation} \label{bound_peri}
P_\gamma(E) \leq  \mathcal{F}(H_s) =  P_\gamma(H_s) + \e\sqrt{\pi/2} |b(H_s)|^2   \leq   \frac{13}{12}e^{\frac{-s^2}{2}}.
\end{equation} 
Let $r$ be such that $\phi(r)= \gamma(E)$. Since $H_r$ maximizes the length of the barycenter
we have by the Gaussian isoperimetric inequality and by \eqref{bound_peri} that   
\[
|b |\leq |b(H_r)| = \frac{1}{\sqrt{2 \pi}} P_\gamma(H_r) \leq \frac{1}{\sqrt{2 \pi}} P_\gamma(E) 
\leq \frac{13}{12\sqrt{2 \pi}}e^{\frac{-s^2}{2}}.
\] 
From our choice of $\e$ in \eqref{epsilon} it follows that 
\begin{equation} \label{bound_bary}
\e |b | \leq \frac{1}{4}.
\end{equation}

Since  $\partial^* E$ is smooth we deduce from the Euler equation \eqref{euler} that for every Lipschitz continuous 
vector field $X : \partial^* E \to \R^n$ with compact support  it holds
\begin{equation} \label{weak_euler}
\int_{\partial^* E} (\div_\tau X  - \la X, x \ra) \, d\Ha^{n-1}_\gamma(x) - \e|b| \int_{\partial^* E}  x_n \la X, \nu \ra \, d\Ha^{n-1}_\gamma(x) 
=  \lambda \int_{\partial^* E}  \la X, \nu \ra \, d\Ha^{n-1}_\gamma(x).
\end{equation}
To obtain \eqref{weak_euler} simply  multiply the Euler equation \eqref{euler} by $\la X, \nu \ra$ and use the divergence 
theorem on hypersurfaces. 

Let $\zeta_k: \partial^* E \to \R$ be the sequence of Lipschitz continuous functions from the proof of 
Lemma~\ref{sing_cap} which have compact support, $0 \leq \zeta_k \leq 1$ and  $\zeta_k \to 1$ in $H_\gamma^1(\partial^* E)$. 
Let us fix  $\unit \in \Sf^{n-1}$ and  choose $X = -\zeta_k^2 x_\unit \unit$ in  \eqref{weak_euler}, where $x_\unit = \la x, \unit \ra$. We use \eqref{bound_bary}, \eqref{weak_euler} and
Young's inequality to get
\begin{equation*}
\begin{split}
\int_{\partial^* E} (x_\unit^2 - (1- \la \nu, \unit \ra^2)) &\zeta_k^2 \, d\gamma(x) -\frac{1}{8} \int_{\partial^* E} (x_\unit^2 + x_n^2)  \zeta_k^2 \, d\gamma(x)\\
&\leq |\lambda|  \int_{\partial^* E} |x_\unit| \zeta_k^2 \, d\gamma(x) + 2 \int_{\partial^* E} \zeta_k |x_\unit|  |D_\tau \zeta_k|  \, d\gamma(x) \\
&\leq  \lambda^2P_\gamma(E) + \frac{1}{2} \int_{\partial^* E} x_\unit^2 \zeta_k^2 \, d\gamma(x) + 4  \int_{\partial^* E}   |D_\tau \zeta_k|^2  \, d\gamma(x).
\end{split}
\end{equation*}
This yields
\[
\frac{3}{8}\int_{\partial^* E} x_\unit^2  \zeta_k^2 \, d\gamma(x) - \frac{1}{8}\int_{\partial^* E} x_n^2  \zeta_k^2 \, d\gamma(x) \leq (\lambda^2 +1)P_\gamma(E)  + 4  \int_{\partial^* E}   |D_\tau \zeta_k|^2  \, d\gamma(x).
\]
Maximizing over $\unit \in \Sf^{n-1}$ gives
\[
\max_{\unit \in \Sf^{n-1}} \left( \frac{1}{4}\int_{\partial^* E} x_\unit^2  \zeta_k^2 \, d\gamma(x) \right) \leq (\lambda^2 +1)P_\gamma(E)  + 4  \int_{\partial^* E}   |D_\tau \zeta_k|^2  \, d\gamma(x).
\]

By letting $k \to \infty$, from the bound $|\lambda| \leq \Lambda$ proved 
in Proposition \ref{1st_variation}, and from \eqref{bound_peri} we deduce 
\begin{equation*}
\max_{\unit \in \Sf^{n-1}} \int_{\partial^* E} \la x, \unit \ra^2 \, d\Ha^{n-1}_\gamma(x) 
\leq 4(\Lambda^2 +1) P_\gamma(E) \leq \frac{13}{3}(\Lambda^2+1)e^{\frac{-s^2}{2}}.
\end{equation*}

\vspace{6pt} 
\noindent \textbf{Step 2}.
In this step we use the previous step and Proposition \ref{2nd_variation} to conclude that for every 
$\varphi \in H_\gamma^1(\partial^* E)$ with $\int_{\partial^* E} \varphi \, d\Ha^{n-1}_\gamma(x) = 0$ it holds
\begin{equation} \label{main_step1}
\int_{\partial^* E} \Bigl(|D_\tau \varphi|^2 - |B_E|^2\varphi^2 - \frac{5}{18}\varphi^2 
- \e |b|  \nu_n \varphi^2\Bigr)\, d\Ha^{n-1}_\gamma(x)  \geq 0.
\end{equation}
Recall that $H_\gamma^1(\partial^* E)$ is the closure of $C_0^\infty(\partial^* E)$ with respect to $H_\gamma^1$-norm. 

Let  $\varphi \in H_\gamma^1(\partial^* E)$ with $\int_{\partial^* E} \varphi \, d\Ha^{n-1}_\gamma(x) = 0$. Then there exists 
$\varphi_k \in C_0^\infty(\partial^* E)$ such that $\varphi_k \to \varphi$ in $H_\gamma^1(\partial^* E)$.
In particular, $\lim_{k \to \infty} \int_{\partial^* E} \varphi_k \, d\Ha^{n-1}_\gamma(x) = 0$. Therefore by slightly changing 
the functions $\varphi_k$ we may assume that they satisfy $\int_{\partial^* E} \varphi_k \, d\Ha^{n-1}_\gamma(x) = 0$ and 
still converge to $\varphi$ in  $H_\gamma^1(\partial^* E)$. Let $\unit_k  \in \Sf^{n-1}$ be  vectors such that 
\begin{equation*}
\Big| \int_{\partial^* E}  \varphi_k \,x \, d\Ha^{n-1}_\gamma(x)\Big|  
= \la \int_{\partial^* E} \varphi_k\, x\, d\Ha^{n-1}_\gamma(x), \unit_k \ra
=  \int_{\partial^* E}  \la x, \unit_k \ra \varphi_k \, d\Ha^{n-1}_\gamma(x).
\end{equation*}
 We use Proposition \ref{2nd_variation} and step 1 to conclude
\begin{equation*}
\begin{split}
\int_{\partial^* E}& \Bigl(|D_\tau \varphi_k|^2 - |B_E|^2\varphi_k^2 - \varphi_k^2 
- \e |b|   \nu_n \varphi_k^2 \Bigr)\, d\Ha^{n-1}_\gamma(x)\\ 
\geq& -\frac{\e}{\sqrt{2\pi}} \left( \int_{\partial^* E} \la x, \unit_k \ra^2 \, d\Ha^{n-1}_\gamma(x)\right)
\left( \int_{\partial^* E} \varphi_k^2 \, d\Ha^{n-1}_\gamma(x)\right)\\
\geq& - \e\,\frac{4}{3\sqrt{2\pi}}(\Lambda^2+1)e^{\frac{-s^2}{2}}   \, \left( \int_{\partial^* E} \varphi_k^2 \, d\Ha^{n-1}_\gamma(x)\right). 
\end{split}
\end{equation*}
From  our choice of $\e$ in  \eqref{epsilon} we conclude that \eqref{main_step1} holds for every $\varphi_k$. 
Since $\varphi_k \to \varphi$ in $H_\gamma^1(\partial^* E)$,  \eqref{main_step1}  
follows by letting $k \to \infty$ and by noticing that  Fatou's lemma implies
\begin{equation*}
\liminf_{k \to \infty} \int_{\partial^* E} |B_E|^2\varphi_k^2\, d\Ha^{n-1}_\gamma(x) \geq \int_{\partial^* E} |B_E|^2\varphi^2\, d\Ha^{n-1}_\gamma(x).
\end{equation*}

Before the next step we remark that by \eqref{main_step1} we have
\begin{equation*}
\int_{\partial^* E}  |B_E|^2\varphi^2 \, d\Ha^{n-1}_\gamma(x) \leq C ||\varphi||^2_{H_\gamma^1(\partial^* E)}
\end{equation*}
for every $\varphi \in H_\gamma^1(\partial^* E)$ with  zero average. Recalling Lemma \ref{sing_cap}  
it is not difficult to see that this implies
\begin{equation} \label{curv_bound}
\int_{\partial^* E}  |B_E|^2\, d\Ha^{n-1}_\gamma(x) < \infty.
\end{equation}
We leave the proof of this estimate  to the reader. 

\vspace{4pt} 
\noindent \textbf{Step 3.}
In this step we will prove that our minimizer $E$ is a  half-space
\begin{equation} \label{main_step2}
E = H_t = \{ x \in \R^n \colon x_n < t \} \quad\text{for some }\; t \in \R.
\end{equation}
This is the main  step of  the proof.

Let $j \in \{1, \dots, n-1\}$. Since we assumed that the barycenter $b(E)$ is in  $- e^{(n)}$ direction, 
the divergence theorem yields 
\begin{equation}\begin{split} \label{test has zero average}
\int_{\partial^* E} &\nu_j\, d\Ha^{n-1}_\gamma(x)
=\frac{1}{(2\pi)^{\frac{n-1}{2}}}\int_E \div( e^{(j)} e^{-\frac{|x|^2}{2}}) \, dx\\
=&-\sqrt{2\pi}\int_E x_j \, d\gamma(x)
=-\sqrt{2\pi}\la b(E), e^{(j)} \ra=0.
\end{split}\end{equation}
In other words, the  function $\nu_j$ has zero average. Moreover \eqref{curv_bound} implies
\begin{equation*}
\int_{\partial^* E}  |D_\tau \nu_j|^2 \, d\Ha^{n-1}_\gamma(x) \leq \int_{\partial^* E}  |B_E|^2 \, d\Ha^{n-1}_\gamma(x) < \infty.
\end{equation*}
From  Lemma \ref{sing_cap} we deduce that $\nu_j \in H_\gamma^1(\partial^* E)$ and we may thus use  \eqref{main_step1} to conclude
\begin{equation} \label{for_nu_j}
\int_{\partial^* E}\Bigl( |D_\tau \nu_j|^2 - |B_E|^2\nu_j^2 - \frac{1}{2}\nu_j^2
- \e |b| \nu_n  \nu_j^2\Bigr) \, d\Ha^{n-1}_\gamma(x)\geq 0.
\end{equation}

Recall the notion  of  tangential derivative $\delta_i$, tangential gradient $D_\tau$ and 
tangential Laplacian $\Delta_\tau$ defined in Section \ref{notation}. 
We recall the well-known equation (see, e.g., \cite[Lemma  10.7]{Giusti})
\begin{equation*}
\Delta_\tau \nu_j = -  |B_E|^2\nu_j + \delta_j \Hh \qquad \text{on }\, \partial^* E.
\end{equation*}
Note also that
\begin{equation*}
\delta_j \langle x, \nu \rangle = \sum_{i=1}^n (\delta_j x_i) \nu_i + (\delta_j \nu_i)x_i = \nu_j -\sum_{i=1}^n \nu_j \nu_i^2 + (\delta_i \nu_j)x_i  = \la D_\tau \nu_j, x \ra,
\end{equation*}
where in the second equality  we  used $\delta_j \nu_i = \delta_i \nu_j$   and 
in the last equality we used $\sum_{i=1}^n \nu_i^2  = |\nu|^2 =1$.  
We differentiate the Euler equation \eqref{euler} with respect to $\delta_j$ and by the two above equations   we deduce that
\begin{equation*}
\Delta_\tau \nu_j - \la D_\tau \nu_j, x \ra= -  |B_E|^2\nu_j  - \e |b | \nu_n \nu_j  \qquad \text{on }\, \partial^* E.
\end{equation*}
The last term follows from $\delta_j x_n = -\nu_j \nu_n$, since $j \neq n$. Let $\zeta_k: \partial^* E \to \R$ be as in step 1.
We multiply the previous equation by $\zeta_k \nu_j$, integrate over $\partial^* E$ 
and use the divergence theorem on hypersurfaces to conclude
\begin{equation*}
\begin{split}
\int_{\partial^* E} \zeta_k &\left(|B_E|^2 \nu_j^2  + \e|b|\nu_n \nu_j^2 \right)\, d\Ha^{n-1}_\gamma(x) 
=-\int_{\partial^* E} \zeta_k \nu_j\left( \Delta_\tau \nu_j - \la D_\tau \nu_j, x \ra \right)\, d\Ha^{n-1}_\gamma(x) \\
&=-\int_{\partial^* E} \zeta_k \nu_j \div_\tau \Bigl( D_\tau \nu_j e^{-\frac{|x|^2}{2}} \Bigr)\, d\Ha^{n-1}(x)\\
&=-\int_{\partial^* E} \div_\tau \Bigl(\zeta_k \nu_j  D_\tau \nu_j e^{-\frac{|x|^2}{2}} \Bigr)\, d\Ha^{n-1}(x) 
+\int_{\partial^* E} \la D_\tau (\zeta_k \nu_j),   D_\tau \nu_j \ra  \, d\Ha^{n-1}_\gamma(x)\\
&=\int_{\partial^* E} \zeta_k | D_\tau \nu_j |^2 \, d\Ha^{n-1}_\gamma(x)
+ \int_{\partial^* E} \nu_j \la D_\tau \zeta_k,D_\tau \nu_j \ra  \, d\Ha^{n-1}_\gamma(x).
\end{split}
\end{equation*}
Since $||D_\tau  \zeta_k||_{L^2(\partial^* E)} \to 0$ as $k \to \infty$, we deduce from the previous equation that
\begin{equation*}
\int_{\partial^* E} \bigl(|B_E|^2 \nu_j^2  +  \e |b| \nu_n \nu_j^2\bigr) \, d\Ha^{n-1}_\gamma(x) 
=\int_{\partial^* E}  | D_\tau \nu_j |^2 \, d\Ha^{n-1}_\gamma(x).
\end{equation*}
Thus we get from \eqref{for_nu_j} that
\begin{equation*}
-\frac{5}{18} \int_{\partial^* E} \nu_j^2 \, d\Ha^{n-1}_\gamma(x)  \geq 0.
\end{equation*}
This implies $\nu_j \equiv 0$ on $\partial^* E$.
Since $E$ has locally finite perimeter in $\R^n$, De Giorgi's structure theorem
 \cite[Theorem 15.9]{Ma} yields
\begin{equation*}
D\chi_E=-\nu\Ha^{n-1}\lfloor\partial^*E.
\end{equation*}
Therefore, the distributional partial derivatives $D_j\chi_E$, $j=1,\ldots, n-1$,
are all zero and necessarily $E=\R^{n-1}\times F$ for some set $F$ of locally
finite perimeter in $\R$. In particular, the topological  boundary of $E$ is smooth and $\partial^* E = \partial E$. 

We will show that the  boundary of $ E$ is connected, which  will imply that    
$E$ is a half-space. To this aim we use the argument from  \cite{StZ}. We argue by contradiction and assume that 
there are  two disjoint closed sets $\Gamma_1, \Gamma_2 \subset \partial E$  such that $\partial E = \Gamma_1 \cup \Gamma_2$. 
Let $a_1 <0 < a_2$ be two numbers such that the function $\varphi: \partial E \to \R$
\begin{equation*}
\varphi := \begin{cases}
a_1,  \quad \text{on }\, \Gamma_1\\
a_2, \quad \text{on }\, \Gamma_2 
\end{cases}
\end{equation*}
has zero average.  Then clearly  $\varphi \in H_\gamma^1(\partial E)$ and, therefore, 
\eqref{main_step1} implies 
\begin{equation*}
\int_{\partial E} \Bigl(|B_E|^2\varphi^2 + \frac{5}{18}\varphi^2 +\e |b|    \nu_n \varphi^2\Bigr) \, d\Ha^{n-1}_\gamma(x)  \leq 0.
\end{equation*}
From  \eqref{bound_bary} we deduce
\begin{equation*}
\int_{\partial E}  \Bigl(|B_E|^2\varphi^2 + \frac{1}{36}\varphi^2\Bigr) \, d\Ha^{n-1}_\gamma(x)  \leq 0
\end{equation*}
which  is obviously  impossible. Hence,   $\partial E$ is connected.

\vspace{4pt} 
\noindent \textbf{Step 4}.
We need  yet to show that $E$ has the correct volume, i.e., $\gamma(E) =\phi(s)$. 
Since we have proved  \eqref{main_step2} we only need to show that the function $f:\R\to (0, \infty)$
\begin{equation*}
f(t) := \mathcal{F} (H_t) = e^{-\frac{t^2}{2}} + \frac{\e}{2\sqrt{2\pi}} e^{-t^2} + \Lambda\sqrt{2\pi}\, \big| \phi(t) - \phi(s) \big|
\end{equation*} 
attains its minimum at $t = s \leq 0$.  

Note that for every $t < 0$ it holds $f(t)< f(|t|)$. Moreover the function $f$ is clearly increasing on $(s,0)$. 
Hence, we only need to show that $f(s) <  f(t) $ for every $t < s$. In $(-\infty,s)$ we have
\begin{equation*}
f'(t) = -t e^{-\frac{t^2}{2}} - \frac{\e}{\sqrt{2\pi}}\, t  e^{-t^2} - \Lambda\, e^{\frac{-t^2}{2}}.
\end{equation*}
It follows from our choices of $\Lambda$ and $\e$ in  \eqref{epsilon} that
$f$ increases, reaches its maximum and decreases to $f(s)$. Moreover, we have
\begin{equation*}
\lim_{t\rightarrow-\infty}f(t)=\Lambda\sqrt{2\pi}\,\phi(s)\geq \sqrt{2}e^{-\frac{s^2}{2}} > f(s).
\end{equation*}
Thus the function $f$ attains its minimum at $t=s$ which implies 
\begin{equation*}
\gamma(E) = \phi(s).
\end{equation*}
This completes the proof.
\end{proof}

\begin{remark} \label{optimal_mass}
We remark that the dependence on the mass  in \eqref{largo gente} is optimal. This can be verified by considering 
the one-dimensional set $E_s = (-\infty, a(s))\cup (-a(s), \infty)$, where $s<0$, and $a(s)<s$ is a number such that
\begin{equation} \label{correct mass}
\frac{2}{\sqrt{2\pi}} \int_{-\infty}^{a(s)} e^{-\frac{t^2}{2}}\, dt 
=\frac{1}{\sqrt{2\pi}} \int_{-\infty}^{s} e^{-\frac{t^2}{2}}\, dt,
\end{equation}
i.e., $\gamma(E_s)= \phi(s)$. Then $b(E_s) = 0$  and $\beta(E_s)= \frac{1}{\sqrt{2 \pi}} e^{-\frac{s^2}{2}}$. 
The sharp mass dependence follows from  
\begin{equation} \label{sharp mass calc}
\liminf_{s \to -\infty} \frac{ D(E_s)}{s^{-2}\,  \beta(E_s)} 
= \frac{1}{\sqrt{2 \pi}} 
\liminf_{s \to -\infty} \frac{2 e^{-\frac{a(s)^2}{2}} - e^{-\frac{s^2}{2}} }{ s^{-2}\, e^{-\frac{s^2}{2}} } 
\leq \sqrt{2/\pi}. 
\end{equation}
For the reader's convenience we will give the calculations below.  
\end{remark}

To show \eqref{sharp mass calc} we write $a(s)= s - \e(s)$.  From \eqref{correct mass} it follows  
that $\e(s) \to 0 $ as $s \to -\infty$. We claim that 
\[
\liminf_{s \to -\infty} \frac{\e'(s)}{s^{-2}} \leq  1.
\]
Indeed, if this were not true then we would have $\e(s)\geq \frac{1}{|s|}$ when  $|s|$ is large. 
Then it follows from \eqref{correct mass} that
\[
\begin{split}
\frac{1}{2} \leq \lim_{s \to -\infty} 
\frac{\int_{-\infty}^{s+1/s} e^{-\frac{t^2}{2}}\, dt}{\int_{-\infty}^{s} e^{-\frac{t^2}{2}}\, dt}
=\lim_{s \to -\infty} \frac{(1-\frac{1}{s^2}) e^{-\frac{(s + 1/s)^2}{2}}}{e^{-\frac{s^2}{2}}} = \frac{1}{e}
\end{split}
\]
which is a contradiction.  By  differentiating \eqref{correct mass} with respect to $s$
and substituting in the left-hand side of \eqref{sharp mass calc} we obtain
\[
\liminf_{s \to -\infty} \frac{2 e^{-\frac{(s -\e(s))^2}{2}} - e^{-\frac{s^2}{2}} }{s^{-2}\, e^{-\frac{s^2}{2}} }   
= \liminf_{s \to -\infty}  \frac{2 \e'(s)\, e^{-\frac{(s -\e(s))^2}{2}}  }{s^{-2}\, e^{-\frac{s^2}{2}} } \leq 2.
\]

\vspace{10pt}
We proceed  by proving that the strong asymmetry controls the square of the standard one. 
Let us introduce a variant of the Fraenkel asymmetry. 
Given a Borel set $E$ with $\gamma(E)=\phi(s)$ we define
\begin{equation*}
\hat{\alpha}(E):=
\begin{cases}
\displaystyle{2 \phi(-|s|) } \hspace{20pt} \text{ if } b(E)=0,\\[4pt]
\displaystyle{\gamma(E\triangle H_{\unit,s})} \;\; \text{ if } b(E)\neq0,
\end{cases}
\end{equation*}
where $\unit =-b(E)/|b(E)|$. Since $\alpha(E) \leq 2 \phi(-|s|) $, then trivially  $\hat{\alpha}(E)\geq\alpha(E)$. 
Compared to the asymmetry $\alpha$, the asymmetry $\hat{\alpha}$ has the advantage  that  
the half-space is chosen to be in the direction of the barycenter. The following estimate can be found in \cite{EL}  
but without explicit constant. We give a proof where we obtain the optimal dependence on the mass.

\begin{proposition}\label{standard vs strong}
Let $E\subset\R^n$ be a set with $\gamma(E)=\phi(s)$. Then
\begin{equation}\label{eq strong vs standard}
\beta(E)\geq \frac{e^{\frac{s^2}{2}}}{4} \,\hat{\alpha}(E)^2.
\end{equation}
\end{proposition}

\begin{proof}
Since $\hat{\alpha}(E)=\hat{\alpha}(\R^n\setminus E)$ we may restrict ourselves to the case $s \leq 0$. 
By first order analysis it is easy to check that the function
\begin{equation*}
f(s):=e^{-\frac{s^2}{2}}-\sqrt{\frac{2}{\pi}}\int_{-\infty}^s e^{-\frac{x_n^2}{2}}dx_n
\end{equation*}
is non-negative in $(-\infty,0]$ or, equivalently, that $e^{-\frac{s^2}{2}}\geq 2\phi(s)$.
Therefore, if $b(E)=0$ we immediately have
\begin{equation*}
\beta(E)=b_s=\frac{e^{-\frac{s^2}{2}}}{\sqrt{2\pi}}\geq\frac{e^{\frac{s^2}{2}}}{\sqrt{2\pi}}\, \hat{\alpha}(E)^2.
\end{equation*}

\vspace{4pt}
Assume now that $b(E)\neq0$ and, without loss of generality, that $e^{(n)}=-b(E)/|b(E)|$.
For simplicity we  write $H=H_{e^{(n)},s}$. Let $a_1$ and $a_2$ be positive numbers such that 
\begin{equation*}
 \gamma(E\setminus H) =\frac{1}{\sqrt{2\pi}} \int_{s-a_1}^s e^{-\frac{x_n^2}{2}} dx_n =  \frac{1}{\sqrt{2\pi}} \int_s^{s+a_2} e^{-\frac{x_n^2}{2}} dx_n.
\end{equation*}
Consider  the sets $E^+:=E\setminus H$, $E^-:=E\cap H$,
$F^+:=\R^{n-1}\times[s,s+a_2)$, $F^-:=\R^{n-1}\times(-\infty,s-a_1)$, and $F:=F^+\cup F^-$.
By construction $\gamma(F)=\phi(s)$, $\gamma(F^+)=\gamma(E^+)$, and $\gamma(F^-)=\gamma(E^-)$.
We have 
\begin{equation*}\begin{split}
\beta(E)-\beta(F)=&\int_E x_n d\gamma(x)-\int_F x_n d\gamma(x)\\
=&\int_{E^+\setminus F^+} (x_n- s-a_2) d\gamma(x)+\int_{F^+\setminus E^+} (-x_n+s+ a_2) d\gamma(x)\\
&+\int_{E^-\setminus F^-} (x_n-s+a_1) d\gamma(x)+\int_{F^-\setminus E^-} (-x_n+s-a_1) d\gamma(x)
\geq0,
\end{split}\end{equation*}
because the integrands in the last term are all positive.

Since $\gamma(E \setminus H) = \gamma(H\setminus E)$ it is sufficient  to show that 
$\beta(F)\geq  e^{\frac{s^2}{2}}\,\gamma(E \setminus H)^2$. 
By first order analysis it is  easy to check that for a fixed $s\leq 0$ the function
\begin{equation*}
g(t):= \int_{s-t}^s (-x_n +s) \,  e^{-\frac{x_n^2}{2}}\, dx_n   
- \frac{e^{\frac{s^2}{2}}}{2}\left(\int_{s-t}^s e^{-\frac{x_n^2}{2}}\, dx_n\right)^2
\end{equation*}
is non-negative in $[0,\infty)$. Indeed, $g'$ is non-negative and $g(0)=0$. By rearranging  terms as above we deduce
\begin{equation*}\begin{split}
\beta(F)&=\int_F x_n d\gamma(x)-\int_H x_n d\gamma(x)\\
&=\frac{1}{(2\pi)^{\frac{n}{2}}}\int_{\R^{n-1}} e^{-\frac{|x'|^2}{2}}dx' 
\left( \int_{s-a_1}^s (-x_n + s) \, e^{-\frac{x_n^2}{2}}dx_n +  \int_s^{s+a_2} (x_n -s) \, e^{-\frac{x_n^2}{2}}dx_n \right) \\
&\geq\frac{1}{\sqrt{2 \pi}} \int_{s-a_1}^s (-x_n+s) e^{-\frac{x_n^2}{2}}dx_n 
\geq  \frac{e^{\frac{s^2}{2}}}{2\sqrt{2\pi}}\left(\int_{s-a_1}^s e^{-\frac{x_n^2}{2}}\, dx_n\right)^2\\
&=\sqrt{\frac{\pi}{2}}\,  e^{\frac{s^2}{2}} \gamma(E \setminus H)^2.
\end{split}\end{equation*}
\end{proof}

\vspace{6pt}
By the Main Theorem and Proposition \ref{standard vs strong} we immediately conclude that
the deficit controls the Fraenkel asymmetry.
\begin{corollary}
There exists an absolute constant $c$ such that for every $s \in \R$ and for every set $E\subset\R^n$  with $\gamma(E)=\phi(s)$ 
the following estimate holds:
\begin{equation}\label{quanti standard new}
\hat{\alpha}(E)^2\leq c\,(1+s^2)e^{-\frac{s^2}{2}} D(E).
\end{equation}
\end{corollary}

\begin{remark}
The reduction to the set $F$ in Proposition \ref{standard vs strong} gives in particular that the 
dependence on the mass in \eqref{eq strong vs standard} is optimal.
We note that even though the dependence on the mass in \eqref{largo gente} and in \eqref{eq strong vs standard} 
are  optimal, we do not know if these together provide the optimal mass dependence for~\eqref{quanti standard new}.  
\end{remark}

\vspace{10pt}
Given a set $E$ of finite Gaussian perimeter, the \emph{excess} of $E$ is defined as
\begin{equation}\label{osci1}
\Os(E):=\min_{\unit \in\Sf^{n-1}}\left\{  
\int_{\partial^*E}|\nu^E- \unit|^2 \, d\Ha^{n-1}_\gamma(x)\right\}.
\end{equation}
We conclude by proving that the isoperimetric deficit controls the excess of the set.
\begin{corollary} \label{osci2}
There exists an absolute constant $c$ such that for every $s \in \R$ and for every set 
of finite  Gaussian perimeter $E\subset\R^n$  with $\gamma(E)=\phi(s)$ 
the following estimate holds:
\begin{equation}\label{osci quanti}
\Os(E)\leq c\,(1+s^2) D(E).
\end{equation}
Moreover, if $b(E)\neq0$, the minimum in \eqref{osci1} is attained by $\unit =-b(E)/|b(E)|$.
\end{corollary}

\begin{proof}
By the divergence theorem
\begin{equation*}\begin{split}
\langle&b(E), \unit \rangle
=\frac{1}{(2\pi)^{\frac{n}{2}}}\int_E\langle x,\unit\rangle e^{-\frac{|x|^2}{2}} dx\\
&=-\frac{1}{(2\pi)^{\frac{n}{2}}}\int_E\div\Bigl(e^{-\frac{|x|^2}{2}}\unit\Bigr) dx
=-\frac{1}{(2\pi)^{\frac{n}{2}}}\int_{\partial^*E}\langle \unit,\nu^E \rangle e^{-\frac{|x|^2}{2}} d\Ha^{n-1}(x)\\
&=\frac{1}{2 \sqrt{2\pi}}\int_{\partial^*E}|\unit-\nu^E|^2 \,  d\Ha^{n-1}_\gamma(x)
-\frac{1}{\sqrt{2\pi}}\int_{\partial^*E}\,   d\Ha^{n-1}_\gamma(x).
\end{split}\end{equation*}
By minimizing over  $\unit \in\Sf^{n-1}$ we get
\begin{equation*}
\Os(E)=2P_\gamma(E)-2\sqrt{2\pi}|b(E)|=2D(E)+2\sqrt{2\pi}\beta(E).
\end{equation*}
Finally, thanks to the estimate \eqref{largo gente}, we obtain \eqref{osci quanti}.
\end{proof}

%%%%%%%%%%%%%%%%%%%%%%%%%%%%%%%%%%%%%%%%%%%%%%%%%%%%%%%%%%%%%%%%%%%%%%%%%%%%%
%%%%%%%%%%%%%%%%%%RINGRAZIAMENTI
\vspace{4pt}
\centerline{\textsc{\large{Acknowledgements}}}

\vspace{4pt}
\noindent
The work was supported by the  FiDiPro project ''Quantitative Isoperimetric Inequalities'' 
and the Academy of Finland grant 268393.

\vspace{4pt}

%%%%%%%%%%%%%%%%%%%%%%%%%%%%%%%%%%%%%%%%%%%%%%%%%%%%%%%%%%%%%%%%%%%%%%%%%%%%%%%%%%%%%%%%%%%%%%%%%%%%%%%

\end{document}